\documentclass[a4paper,reqno]{amsart}

\usepackage{amsmath,amsthm, amssymb}
\usepackage{mathrsfs}
\usepackage[shortlabels]{enumitem}
\usepackage{graphicx}
\usepackage[font=small]{caption}
\usepackage{xcolor}
\usepackage{url}

\newcommand{\N}{\mathbb{N}}
\newcommand{\R}{{\mathbb{R}}}
\newcommand{\C}{{\mathbb{C}}}

\newcommand{\dd}{{{\rm d}}}
\newcommand{\ii}{{\rm i}}

\newcommand{\ie}{{\emph{i.e.}}}
\newcommand{\eg}{{\emph{e.g.}}}

\newcommand{\la}{\lambda}

\newcommand{\eps}{\varepsilon}

\newcommand{\Dom}{{\operatorname{Dom}}}

\renewcommand{\Re}{\operatorname{Re}}
\renewcommand{\Im}{\operatorname{Im}}
\newcommand{\dist}{\operatorname{dist}}
\newcommand{\sgn}{\operatorname{sgn}}
\newcommand{\supp}{\operatorname{supp}}

\newcommand{\loc}{\mathrm{loc}}
\newcommand{\BigO}{\mathcal{O}}

\newcommand{\opH}{H}

\newcommand{\opA}{A}

\newcommand{\opI}{I}

\newcommand{\opS}{S}
\newcommand{\opT}{T}
\newcommand{\opU}{U}
\newcommand{\opZ}{Z}

\newcommand{\Rplus}{\R_+}

\newcommand{\intR}{\int_{\R}}

\newcommand{\Lt}{{L^2}}

\newcommand{\LilocR}{{L^{\infty}_{\rm loc}(\R})}

\newcommand{\CcR}{{C_c^{\infty}(\R)}}

\newcommand{\WotR}{{W^{1,2}(\R)}}

\newcommand{\WttR}{{W^{2,2}(\R)}}

\newcommand{\Dt}{-\partial_x^2}

\newcommand{\Dtp}{\partial_x^2}
\newcommand{\Nt}{-\partial_x}

\newcommand{\Ntp}{\partial_x}

\newcommand{\Ntime}{\partial_t}

\newcommand{\ls}{\lesssim 	}
\newcommand{\gs}{\gtrsim}

\theoremstyle{plain}

\newtheorem{theorem}{Theorem}[section]
\newtheorem{lemma}[theorem]{Lemma}

\newtheorem{proposition}[theorem]{Proposition}

\theoremstyle{definition}
\newtheorem{example}[theorem]{Example}
\newtheorem{remark}[theorem]{Remark}
\newtheorem{asm-sec}[theorem]{Assumption}

\newcommand\cH{\mathcal H}

\newcommand\cS{\mathcal S}

\newcommand\sF{\mathscr F}
\newcommand\sK{\mathscr K}
\newcommand\sL{\mathscr L}

\usepackage{mathtools}
\mathtoolsset{showonlyrefs}

\numberwithin{equation}{section}
\numberwithin{figure}{section}

\begin{document}
\title[Generalised Airy operators]{Generalised Airy operators}

\author{Antonio Arnal}

\address[Antonio Arnal]{Mathematical Sciences Research Centre, Queen's University Belfast, University Road, Belfast BT7 1NN, UK}

\email{aarnalperez01@qub.ac.uk}

\author{Petr Siegl}

\address[Petr Siegl]{Institute of Applied Mathematics, Graz University of Technology, Steyrergasse 30, 8010 Graz, Austria}

\email{siegl@tugraz.at}




\date{\today}

\begin{abstract}
	We study the behaviour of the norm of the resolvent for non-self-adjoint operators of the form $\opA := \Nt + W(x)$, with $W(x) \ge 0$, defined in $\Lt(\R)$. We provide a sharp estimate for the norm of its resolvent operator, $\| (\opA - \la)^{-1} \|$, as the spectral parameter diverges $(\la \to +\infty)$. Furthermore, we describe the $C_0$-semigroup generated by $-\opA$ and determine its norm. Finally, we discuss the applications of the results to the asymptotic description of pseudospectra of Schr\"odinger and damped wave operators and also the optimality of abstract resolvent bounds based on Carleman-type estimates.
\end{abstract}

\maketitle

\section{Introduction}
\label{sec:intro}
The complex Airy operator in $\Lt(\R)$
\begin{equation}\label{eq:Airy.intro}
	\opA_0 = \Dt + i x, \quad \Dom(\opA_0) = \WttR \cap \Dom(x),
\end{equation}
and the corresponding realisations in $\Lt(\Rplus)$ and in $\Lt((a, b))$ appear in various contexts, \eg~in the study of resonances \cite{Herbst-1979-64}, superconductivity~\cite{Almog-2008-40,Almog-2010-300,Rubinstein-2010-195}, the Bloch-Torrey equation~\cite{grebenkov2017complex,Grebenkov-2018-50}, MRI \cite{Grebenkov-2018-269}, hydrodynamics \cite{Shkalikov-2004-124}, control theory~\cite{Beauchard-2015-21}, spectral instability \cite{Henry-2012-350}, spectral approximation in domain truncations \cite{Semoradova-toappear}, asymptotic description of pseudospectra of Schr\"odinger operators \cite{ArSi-resolvent-2022}, completeness properties of the eigensystem and similarity to a normal operator \cite{Savchuk-2017-51, Tumanov-2017-96}, momenta with complex magnetic fields, quasi-self-adjoint Quantum Mechanics and Hardy-type inequalities in $\R^2$ \cite{Krejcirik-2019-51} - see also \cite{Helffer-NYUS-2016}, \cite[Sec.~14.3, Sec.~14.5]{Helffer-2013-book}.

The crucial aspect of \eqref{eq:Airy.intro} is that it represents a tractable non-self-adjoint differential operator for which more explicit methods can be applied. In addition, more complicated operators in certain asymptotic regimes (\eg~the semi-classical one or those with large spectral parameter) can often be reduced to the analysis of the Airy operator. At the same time, the properties of $A_0$ already exhibit many striking and ``highly non-self-adjoint'' features. In particular:
\begin{enumerate}[\upshape (i), wide]
	\item $A_0$ is m-accretive, has compact resolvent \textit{and} $\sigma(\opA_0) = \emptyset$;
	\item the resolvent norm only depends on $\Re z$ and grows super-exponentially at $+\infty$
	\begin{equation}
		\label{eq:complex.airy.res.norm}
		\| (\opA_0 - \Re z)^{-1} \| = \sqrt{\tfrac{\pi}{2}} (\Re z)^{-\frac14} \exp\left(\tfrac43 (\Re z)^{\frac32} \right) (1 + o(1)), \quad \Re z \to +\infty;
	\end{equation}
	\item the semigroup $S_t$ generated by $-A_0$ can be determined explicitly and decays super-exponentially
	\begin{equation*}
		\|S_t\| = \exp( - \tfrac{t^3}{12}), \quad t >0.
	\end{equation*}
\end{enumerate}
For details, see \eg~\cite[Sec.~14.3, Sec.~14.5]{Helffer-2013-book}. The asymptotic equality \eqref{eq:complex.airy.res.norm} was proved 
in \cite[Cor.~1.4]{BordeauxMontrieux-2013} using pseudo-differential operator techniques within a semi-classical framework pioneered in \cite{Dencker-2004-57}. It improved an earlier estimate, presented in \cite{martinet-these} (see also \cite[Prop.~2.10]{BordeauxMontrieux-2013} or \cite[Prop.~14.11]{Helffer-2013-book}), which was derived via an analysis of the corresponding semigroup (see \cite[Sec.~1]{BordeauxMontrieux-2013}).

One of the reasons why the precise results summarized above are available is that, by transforming $\opA_0$ to Fourier space (via $\sF \opA_0 \sF^{-1}$, where $\sF$ denotes the Fourier transform), one obtains the first order operator
\begin{equation*}
	\opA_2 = \Nt + x^2, \quad \Dom(\opA_2) = \WotR \cap \Dom(x^2)
\end{equation*}
(where we have retained $x$ to denote the independent variable). 
In this paper, our main objective is to study generalised Airy operators in $\Lt(\R)$
\begin{equation}\label{eq:gen.A.intro}
\opA = \Nt + W(x),
\end{equation}
where $W$ is a non-negative, even, eventually increasing function which is unbounded at infinity, see Assumption~\ref{asm:W.airy} for details. Operators \eqref{eq:gen.A.intro} share many properties of the complex Airy operator \eqref{eq:Airy.intro}. In particular, they are m-accretive, have compact resolvent and empty spectrum, see~Proposition~\ref{prop:airy.basic.properties} proved in \cite[App.~A]{ArSi-resolvent-2022}. Our goal here is to analyse precisely the resolvent and semigroup norms.  

Our motivation originates in the analysis of pseudospectra of Schr\"odinger operators with complex potentials in \cite{ArSi-resolvent-2022} and of damped wave operators with unbounded damping in \cite{arnal2022resolvent-dwe}, where the resolvent norm of $A$ enters the asymptotic formulas for the level-curves of the pseudospectra, see more details below and in Sub-section~\ref{ssec:level.curves}. Nonetheless operators \eqref{eq:gen.A.intro} furnish particularly simple instances, tractable by relatively elementary methods, that already exhibit the sort of pathologies found in more complex non-self-adjoint operators. Moreover the flexibility in the rate of growth of $W$ at infinity permits a wide range of resolvent and semigroup norm behaviours (see examples below and Examples~\ref{ex:W.resolvent}, \ref{ex:W.sg}). These can in turn be used to address the optimality of abstract resolvent upper bounds based on Carleman-type estimates (see details below and in Sub-section~\ref{ssec:bandtlow.optimal}).

Our main finding, Theorem~\ref{thm:airy.norm}, states that, if $W$ satisfies Assumption~\ref{asm:W.airy}, then
\begin{equation}\label{eq:A.res.intro}
	\| (\opA - \la)^{-1} \| = \sqrt{\pi} (W'(x_{\la}))^{-\frac12} \exp \left(2 f_{\la}(x_{\la})\right)  (1 + o(1)), \quad \la \to +\infty,
\end{equation}
with $x_{\la} > 0$ the (positive) turning point of $W$ defined by $W(x_{\la}) = \la$ and $f_{\la}$ the primitive of $W$ defined by 
\begin{equation*}
	f_{\la}(x) := \int_{0}^{x} (\la - W(t)) \dd t, \quad x \in \R; 
\end{equation*}
note that, as for $A_0$, the resolvent norm is independent of $\Im \la$. 

This result is in agreement with \eqref{eq:complex.airy.res.norm}, where $W(x) = x^2$, and it yields a variety of rates for other examples of $W$, namely
as $\la \to +\infty$,
\begin{align*}
	W(x) &= \log \langle x \rangle, & \| (\opA - \la)^{-1} \| &= \sqrt{\pi} \exp\left(2 \sqrt{e^{2 \la} - 1} + \la/2 - \pi\right) (1 + o(1)),\\
	W(x) &= |x|^p, & \| (\opA - \la)^{-1} \| &= \sqrt{\tfrac{\pi}{p}} \la^{\frac{1 - p}{2 p}} \exp\left(\tfrac{2 p}{p + 1} \la^{\frac{1 + p}{p}}\right) (1 + o(1)), \ p>0, \\
	W(x) &= e^{|x|}, & \| (\opA - \la)^{-1} \| &= \sqrt{\pi} \la^{-\frac12} \exp\left(2 \la \log(\la) (1 + o(1))\right) (1 + o(1));
\end{align*}
see Example~\ref{ex:W.resolvent}. It is apparent that the \emph{more slowly} $W$ grows the larger the resolvent will be: in the above examples, the largest resolvent corresponds to $W(x) = \log \langle x \rangle$. A similar behaviour can be observed for one-dimensional Schr\"odinger operators with complex potentials (see \cite[Sec.~7]{ArSi-resolvent-2022}). 

In addition, in Theorem~\ref{thm:airy.sg.norm}, we find the $C_0$-semigroup $\opS_t$ generated by $-\opA$ and  calculate its norm. The latter decays in direct proportion to the growth of $W$, \eg~
\begin{align*}
	W(x) &= \log \langle x \rangle, & \| \opS_t \| &= \langle \tfrac t2 \rangle^{-t} \exp\left(t - \pi\right) (1 + o(1)), & &t \to +\infty,\\
	W(x) &= |x|^p, & \| \opS_t \| &= \exp \left(-\tfrac{2^{-p}}{p + 1} t^{p + 1}\right), & &t \ge 0, \ p>0,\\
	W(x) &= e^{|x|}, & \| \opS_t \| &= \exp\left(-2 e^{\frac{t}{2}} (1 + o(1))\right), & &t \to +\infty,
\end{align*}
\ie~rapidly growing functions such as $W(x) = e^{|x|}$ show faster semigroup decay; see Example~\ref{ex:W.sg} for more details.

As indicated above, the resolvent norm of $A_0$ and that of $A$ appear in the asymptotic formulas for the shape of the pseudospectra of second order operators. In \cite[Prop.~5.1]{ArSi-resolvent-2022}, it was shown that, for a wide class of Schr\"odinger operators with complex potential unbounded at infinity, $\opH = \Dt + i V$, in $\Lt(\Rplus)$ (see \cite[Asm.~3.1]{ArSi-resolvent-2022}), the norm of the resolvent of $\opH$ along curves $\la_b = a(b) + i b$ inside the numerical range and adjacent to the imaginary axis can be estimated precisely as a function of the norm of the resolvent of $\opA_0$
\begin{equation*}
	\|(\opH - \la_b)^{-1} \| = \| (\opA_0 - \mu_b)^{-1} \| \left( V'(x_b) \right)^{-\frac{2}{3}} (1 + o(1)), \quad b \to +\infty,
\end{equation*}
where $x_b > 0$ is the turning point of $V$ determined by $V(x_b) = b$ and $\mu_b = a(b) ( V'(x_b))^{-\frac{2}{3}}$. Employing \eqref{eq:complex.airy.res.norm}, this in turn yields the asymptotic shape of the $\eps$-pseudospectra (with $\eps>0$)
\begin{equation}
	\label{eq:resnorm.iR.levelcurves}
	a = \left(\tfrac34\right)^{\frac23} V'(x_b)^\frac23 (\log(V'(x_b)^\frac23 \eps^{-1}))^{\frac23} (1 + o(1)), \quad b \to +\infty;
\end{equation}
see \cite[Sub-sec.~5.1.1]{ArSi-resolvent-2022} for details and also \cite{BordeauxMontrieux-2013} for results on semi-classical operators.

Operators $A$ with $W(x) = |x|^p$, $p>0$, which we denote by $\opA_p := \Nt + |x|^p$, play a similar role to that outlined above for $\opA_0$ when estimating the norm of the resolvent for Schr\"odinger operators $\opH = -\partial_x^2 + i V$ in $\Lt(\R)$ where $V$ is even, unbounded at infinity and regularly varying with index of variation equal to $p$ (see \cite[Asm.~4.1]{ArSi-resolvent-2022}). Along curves $\la_a = a + i b(a)$  inside the numerical range and adjacent to the real axis (see \cite[Prop.~5.2]{ArSi-resolvent-2022}), we have 
\begin{equation}
	\label{eq:schrod.re.gral.res.norm}
	\| (\opH - \la_a)^{-1} \| = \| (\opA_p - \mu_a)^{-1} \| (V(t_a))^{-1} ( 1 + o(1)), \quad a \to +\infty,
\end{equation}
where $t_a > 0$ is the solution of $t_a V(t_a) = 2 \sqrt{a}$ and $\mu_a = b(a) (V(t_a))^{-1}$. Combining \eqref{eq:schrod.re.gral.res.norm} and \eqref{eq:A.res.intro}, one finds the asymptotic shape of the $\eps$-pseudospectra (with $\eps > 0$)
\begin{equation*}
	b = \left(\tfrac{p + 1}{2p}\right)^{\frac{p}{p + 1}} V(t_a) \left(\log(V(t_a) \eps^{-1})\right)^{\frac{p}{p + 1}} (1 + o(1)), \quad a \to +\infty,
\end{equation*}
see Sub-section~\ref{sssec:schrodinger.level.curves} and \cite[Sub-sec.~5.1.2]{ArSi-resolvent-2022} for details.

Finally, operators of the general form $\opA := \Nt + a(x)$ also emerge in the study of the one-dimensional wave equation with unbounded damping $a(x) \ge 0$. It has been shown in \cite[Prop.~4.12]{arnal2022resolvent-dwe} that, for curves $\la_b = -c(b) + i b$ in the second and third quadrants of the complex plane adjacent to the imaginary axis, the associated quadratic operator family $\opT(\la) = \Dt + q(x) + 2 \la a(x) + \la^2$, with $a, q$ obeying \cite[Asm.~3.1]{arnal2022resolvent-dwe}, satisfies
\begin{equation}
	\label{eq:dwe.quad.res.norm}
	\| \opT(\la_b)^{-1} \| = \| (\opA - c(b))^{-1} \| (2 |b|)^{-1} (1 + o(1)), \quad |b| \to +\infty.
\end{equation}
As a consequence, one can analogously describe the level-curves $\| (\opT(\la_b))^{-1} \| = \eps^{-1}$ with $\eps > 0$. For instance, in the case where $q(x)=0$ and $a(x)= \log \langle x \rangle^p$, $p>0$, one obtains the asymptotic shape
\begin{equation*}
	c = p \log\log(2 b \eps^{-1}) (1 + o(1)), \quad b \to +\infty,
\end{equation*}
see Sub-section~\ref{sssec:dwe.level.curves} and \cite[Sub-sec.~4.2]{arnal2022resolvent-dwe} for details and another example with $a(x) = x^{2n}$, $n \in \N$.

The operators \eqref{eq:gen.A.intro} with $W$ satisfying Assumption~\ref{asm:W.airy} have compact resolvent (see Proposition~\ref{prop:airy.basic.properties}). Moreover, depending on the rate of growth of $W$, they belong to certain ideals of compact operators, \eg~Schatten classes, see Sub-section~\ref{ssec:bandtlow.optimal}. Estimates of the norm of the resolvent for elements in trace ideals of compact operators  go back to the early 20th century with the work of T. Carleman on integral operators. More recently, upper bounds have been found for (abstract) non-normal compact operators in the context of the analysis of their spectral behaviour under perturbation (see \cite{Bandtlow-2004-267, SARIHAN2021QUANTITATIVE, bandtlow2015explicit, bandtlow2008resolvent}). We shall review these bounds in Sub-section~\ref{ssec:bandtlow.optimal} in the light of our own findings. It is remarkable that 
the simple differential operators \eqref{eq:gen.A.intro} (and also one dimensional Schr\"odinger operators with imaginary monomial potentials) are close to exhausting the abstract bounds which do not take into account any particular structure of the operator besides the singular values of the resolvent.

Unlike some of the research mentioned earlier, our methods are relatively elementary. The two main tools used in the proof of Theorem~\ref{thm:airy.norm} are Schur's test with weights (see Sub-section~\ref{ssec:schur.prelim}) and an extension of Laplace's method for integrals depending on a large parameter (see Lemma~\ref{lem:laplace.rplus}). The proof of \eqref{eq:A.res.intro} is structured in three steps:
\begin{enumerate}[wide]
	\item in Proposition~\ref{prop:airy.norm.ubound}, we use Schur's test to prove that a certain integral operator, $\opT_{k_{\la}}$ (see \eqref{eq:Tk.def} and \eqref{eq:kla.def}), is bounded in $\Lt$ and we find an upper bound;
	\item in Proposition~\ref{prop:airy.norm.lbound}, we show that the upper bound found in the previous step is optimal;
	\item in the last step, the theorem's proof, it is shown that $\opT_{k_{\la}}$ is in fact the resolvent operator $(\opA - \la)^{-1}$.
\end{enumerate}

The remainder of the paper is structured as follows. Section~\ref{sec:prelim} introduces our notation and recalls some fundamental tools applied later on. Section~\ref{sec:assumptions} states our assumptions. Section~\ref{sec:resolvent.norm} is devoted to formulating and proving our estimate for the norm of the resolvent. Section~\ref{sec:sg.norm} describes the $C_0$-semigroup generated by $-\opA$ and calculates its norm. In Section~\ref{sec:further.remarks}, we apply our findings in two different contexts. Firstly (Sub-section~\ref{ssec:level.curves}), we show the level curves for the norm of the resolvent in the particular cases outlined earlier (Schr\"odinger operators with complex potentials and the quadratic operator family associated with the damped wave equation). Secondly (Sub-section~\ref{ssec:bandtlow.optimal}), we show that upper bounds for the resolvent of elements in certain classes of compact operators derived in \cite{Bandtlow-2004-267,SARIHAN2021QUANTITATIVE} using abstract methods are in fact almost optimal.

\section{Notation and preliminaries}
\label{sec:prelim}
We write $\Rplus := (0, +\infty)$ and the characteristic function of a set $E$ is denoted by $\chi_E$. We use $\CcR$ to represent the space of smooth functions with compact support. Unless otherwise stated, our underlying Hilbert space shall be $\Lt(\R)$. The $L^2$ inner product shall be denoted by $\langle \cdot, \cdot \rangle$ and the $L^2$ norm by $\| \cdot \|$. In the one-dimensional setting, we will denote the first and second order differential operators by $\Ntp$ and $\Dtp$, respectively. We will also appeal to the Sobolev spaces $\WotR$ and $\WttR$ (see \eg~\cite[Sub-sec.~V.3]{EE} for definitions).

If $\cH$ is a Hilbert space, $\sL(\cH)$ shall denote the (Banach) space of bounded linear operators on $\cH$. For a closed, densely defined linear operator $\opT$ on a Hilbert space $\cH$, we will write as usual $\rho(\opT)$, $\sigma(\opT)$ and $\sigma_{p}(\opT)$ to denote its resolvent set, its spectrum and the set of its eigenvalues, respectively.

If $\cH$ is a Hilbert space, we will represent by $\sK(\cH)$ the set of compact operators on $\cH$, which is a closed two-sided ideal in $\sL(\cH)$. For $0 < p < \infty$, we shall denote by $\cS_p$ the Schatten-von Neumann class (or just Schatten class)
\begin{equation*}
	\cS_p := \left\{\opT \in \sK(\cH) : \| \opT \|_p := \left(\sum_{k = 1}^{\infty} (s_k(\opT))^p\right)^{\frac1{p}} < \infty\right\},
\end{equation*}
where $s_k(\opT)$ represents the $k$th singular value of $\opT$ (such values are assumed to be listed in decreasing order and repeated according to their multiplicity). We refer the interested reader to \eg~\cite{DS2} or \cite{Simon-2005} for details on the theory of these classes.

To avoid introducing multiple constants whose exact value is inessential for our purposes, we write $a \ls b$ to indicate that, given $a,b \ge 0$, there exists a constant $C>0$, independent of any relevant variable or parameter, such that $a \le Cb$. The relation $a \gs b$ is defined analogously whereas $a \approx b$ means that $a \lesssim b$ \textit{and} $a \gs b$.

\subsection{Basic properties of generalised Airy operators}
\label{ssec:baisc.prelim}
It has been shown (see \cite[App.~A]{ArSi-resolvent-2022}) that the following result holds, \ie~$A$ shares the basic properties with the complex Airy operator $A_0$.

\begin{proposition}
	\label{prop:airy.basic.properties}
	Let $W \in \LilocR \cap C^1(\R\setminus [-x_0,x_0])$, for some $x_0 > 0$, and
	\begin{equation}
		\label{eq:genairy.def}
		\opA := \Nt + W, \quad
		\Dom(\opA) := \{ u \in \Lt(\R) \, : \, -u' + W u \in \Lt(\R) \}.
	\end{equation}
	Assuming $W \ge 0$ a.e. and $\lim_{|x| \to +\infty} W(x) = +\infty$, then
	\begin{enumerate}[\upshape (i), wide]
		\item \label{itm:A.maccr} $A$ is densely defined and m-accretive;
		\item $A$ has compact resolvent and $\sigma(\opA) = \emptyset$;
		\item \label{itm:A.star} the adjoint operator reads
		\begin{equation}
			\opA^* = \Ntp + W, \quad
			\Dom(\opA^*) = \{ u \in \Lt(\R) \, : \, u' + W u \in \Lt(\R) \}.
		\end{equation}
	\end{enumerate}
	If, in addition, there exist $\eps \in (0,1)$ and $M> 0$ such that
	\begin{equation}
		\label{eq:A.W.sep.asm}
		|W'(x)| \le \eps |W(x)|^2 + M, \quad |x| > x_0,
	\end{equation}
	then
	\begin{equation}
		\label{eq:A.dom.sep}
		\Dom(\opA) = \Dom(A^*) = W^{1,2}(\R) \cap \Dom(W)
	\end{equation}
	and we have
	\begin{equation}
		\label{eq:genAiry.graph.norm}	
		\begin{aligned}
			\| \opA u \|^2 + \| u \|^2 &\geq C_A \left(\| u' \|^2 + \| W u \|^2 + \| u \|^2 \right), \quad u \in \Dom(A),
			\\
			\| \opA^* u \|^2 + \| u \|^2 &\geq C_{A^*} \left(\| u' \|^2 + \| W u \|^2 + \| u \|^2 \right), \quad u \in \Dom(A^*);
		\end{aligned}
	\end{equation}
	the constants $C_{A}, C_{A^*}>0$ depend only on $\eps$, $M$ and $\|W \chi_{[-x_0,x_0]}\|_\infty$.
\end{proposition}

\subsection{Schur's test}
\label{ssec:schur.prelim}
Let $\opT_k$ be the integral operator determined by the non-negative kernel $k: \R \times \R \to [0, +\infty)$ \ie~formally
\begin{equation}
	\label{eq:Tk.def}
	\opT_k u(x) = \intR k(x,y) u(y) \dd y, \quad x \in \R.
\end{equation}
Schur's test states that, if there exist strictly positive measurable functions $p, q$ on $\R$ and constants $\alpha, \beta > 0$ such that $k(x,y)$ satisfies
\begin{equation}
	\label{eq:schur.assm.2}
	\begin{aligned}
		&\intR k(x,y) q(y)\dd y \le \alpha p(x) \quad \text{for almost every } x \in \R,\\
		&\intR k(x,y) p(x) \dd x \le \beta q(y) \quad \text{for almost every } y \in \R,
	\end{aligned}
\end{equation}
then the operator $\opT_k$ in \eqref{eq:Tk.def} is well-defined for every $u \in \Lt(\R)$ and
\begin{equation}
	\label{eq:schur.ubound}
	\| \opT_k \| \le \sqrt{\alpha \beta}
\end{equation}
(see \eg~\cite[Thm.~5.2]{Halmos-1978-96}).

\section{Assumptions}
\label{sec:assumptions}
We begin by listing the assumptions that $W$ will obey throughout the rest of the paper.

\begin{asm-sec}
	\label{asm:W.airy}
	Suppose that $W \in \LilocR \cap C^2(\R \setminus [-x_0, x_0])$, for some $x_0 \in \Rplus$, with $W \ge 0$ a.e. and assume further that the following conditions are satisfied:
	\begin{enumerate} [\upshape (i)]
		\item \label{itm:W.even} $W$ is even:
		\begin{equation*}
			W(-x) = W(x), \quad x \in \R;
		\end{equation*}
		\item \label{itm:W.incr.unbd} $W$ is unbounded and eventually increasing in $\Rplus$:
		\begin{equation}
			\label{eq:W.unbd.incr}
			\lim_{x \to +\infty} W(x) = +\infty, \quad W'(x) > 0, \quad x > x_0;
		\end{equation}
		\item \label{itm:W.control.derv} $W$ has controlled derivatives: there exist $\nu \in [-1, +\infty)$ such that
		\begin{equation}
			\label{eq:W.control.derv}
			W'(x) \ls W(x) x^{\nu}, \quad
			|W''(x)| \ls W'(x) x^{\nu}, \quad x > x_0;
		\end{equation}
		\item \label{itm:W.upsilon} $W$ grows sufficiently fast: we have
		\begin{equation}
			\label{eq:W.Ups.decay}
			\Upsilon_1(x) = o(1), \quad |x| \to +\infty,
		\end{equation}
		where
		\begin{equation}
			\label{eq:W.ups.def}
			\Upsilon_1(x) := |x|^{\nu} |W'(x)|^{-\frac12},\quad |x| > x_0.
		\end{equation}
	\end{enumerate}
\end{asm-sec}

\begin{remark}
	\label{rmk:asm}
	We make the following observations based on the above conditions.
	\begin{enumerate}[\upshape (i), wide]
		\item \label{itm:asm.dom.separation} From Assumptions~\ref{asm:W.airy}~\ref{itm:W.control.derv},~\ref{itm:W.upsilon}, we have
		\begin{equation*}
			\frac{|W'(x)|}{(W(x))^2} \ls \frac{|W'(x)|}{|W'(x)|^2 |x|^{-2 \nu}} = (\Upsilon_1(x))^2 = o(1), \quad |x| \to +\infty.
		\end{equation*}
		It follows that \eqref{eq:A.W.sep.asm} holds and therefore the domain of the corresponding generalised Airy operator separates
		\begin{equation}
			\label{eq:airy.op.def}
			\opA = \Nt + W(x), \quad \Dom(\opA) = \WotR \cap \Dom(W).
		\end{equation}
		\item \label{itm:asm.W.variation.cte} Assumption~\ref{asm:W.airy}~\ref{itm:W.control.derv} implies that, for any $0 < \eps < 1$, all sufficiently large $|x|$ and $|\delta| \le \eps |x|^{-\nu}$, we can control the variation of $W$ and that of $W'$ on intervals whose length is of order $|x|^{-\nu}$
		\begin{equation}
			\label{eq:W.approx.cte}
			\frac{W^{(j)}(x + \delta)}{W^{(j)}(x)} \approx 1, \quad j \in \{0, 1\},
		\end{equation}
		(see \cite[Lem.~4.1]{MiSiVi-2020}).
		\item \label{itm:asm.upsilon} The quantity $\Upsilon_1(x)$ in Assumption~\ref{asm:W.airy}~\ref{itm:W.upsilon} is the equivalent, for a first order differential operator such as $\opA$, of $\Upsilon(x)$ for a Schr\"odinger operator in \cite[Asm.~3.1~(iii)]{ArSi-resolvent-2022}.
	\end{enumerate}
\end{remark}

\begin{example}
	\label{ex:W.basic}
	Assumption~\ref{asm:W.airy}~\ref{itm:W.control.derv} encompasses a wide range of unbounded $W$ growing at very different rates, for example:
	\begin{enumerate}[\upshape (i), wide]
		\item logarithmic functions such as $W(x) = \log \langle x \rangle^p$, $p > 0$;
		\item \label{itm:W.poly.basic}polynomial functions such as $W(x) = |x|^p$, $p > 0$ - note that, when $p = 2$, we recover, via the Fourier transform, the complex Airy operator $\opA_0$ (see the comments to this effect in Section~\ref{sec:intro});
		\item exponential functions such as $W(x) = e^{|x|^p}$, $p > 0$.
	\end{enumerate}
\end{example}

\section{The norm of the resolvent operator}
\label{sec:resolvent.norm}
In this section, our aim is to provide an estimate of $\| (\opA - \la)^{-1} \|$ as $\la \to +\infty$. Note that, since $\opA$ is m-accretive, it is immediate that $\| (\opA - \la)^{-1} \| \le 1/|\la|$ for every $\la < 0$. Furthermore, it is sufficient to consider $\la \in \R$ since, for general $\la \in \C$, we have $(\opA - \la)^{-1} = \opU_{\Im \la}^{-1} (\opA - \Re\la)^{-1} \opU_{\Im \la}$, with the isometry $\opU_{a} u(x) = e^{i a x} u(x)$ for any $u \in \Lt(\R)$ and $a \in \R$, and hence $\| (\opA - \la)^{-1} \| = \| (\opA - \Re\la)^{-1} \|$.

\subsection{Statement of the result}
\label{ssec:norm.resolvent.statement}
By Assumption~\ref{asm:W.airy}~\ref{itm:W.incr.unbd}, there exists $\la_0 > 0$ (\eg~take any $\la_0 > \| W \chi_{[0, x_0]} \|_{\infty}$) such that, for every $\la \ge \la_0$, the equation $W(x) = \la$ has a unique solution in $\Rplus$ which we shall denote by $x_{\la}$ (the \textit{turning point} of $W$), \ie
\begin{equation}
	\label{eq:W.tp.def}
	W(x_{\la}) = \la, \qquad x_\la >0.
\end{equation}
Note that
\begin{equation}
	\label{eq:xlapm.unbounded}
	\lim_{\la \to +\infty} x_{\la} = +\infty.
\end{equation}

Furthermore, for every $\la > 0$ we define the real function
\begin{equation}
	\label{eq:fla.def}
	f_{\la}(x) := \int_{0}^{x} (\la - W(t)) \dd t, \quad x \in \R.
\end{equation}
Clearly $f_{\la} \in W^{1,\infty}_{\loc}(\R) \cap C^3(\R \setminus [-x_0,x_0])$, $f_{\la}(0) = 0$ and
\begin{equation}
	\label{eq:fla.deriv}
	f'_{\la}(x) = \la - W(x), \quad \text{a.e. in } \R, \quad f''_{\la}(x) = -W'(x), \quad x \in \R \setminus [-x_0, x_0].
\end{equation}

An elementary analysis provides the following properties for $f_{\la}$.
\begin{lemma}
	\label{lem:fla.prop}
	For $\la > 0$ and $f_{\la}$ as defined by \eqref{eq:fla.def}, we have that $f_{\la}$ is an odd function and in addition
	\begin{enumerate}[\upshape (i)]
		\item \label{itm:fla.unbded} $f_{\la}$ is unbounded
		\begin{equation}
			\label{eq:fla.unbounded}
			\lim_{x \to +\infty} f_{\la}(x) = -\infty.
		\end{equation}
	\end{enumerate}
	Furthermore there exists $\la_0 > 0$ such that for all $\la \ge \la_0$
	\begin{enumerate}[\upshape (i)]
		\setcounter{enumi}{1}
		\item \label{itm:fla.cps} The value $x_{\la}$ defined by \eqref{eq:W.tp.def} is a critical point
		\begin{equation}
			\label{eq:fla.critical.points}
			f'_{\la}(x_{\la}) = 0, \quad f''_{\la}(x_{\la}) < 0.
		\end{equation}
		\item \label{itm:fla.zeroes} There exists $x_{\la,0} > x_{\la}$ (which is unique in $\Rplus$) such that
		\begin{gather}
			\label{eq:fla.zeroes} f_{\la}(x_{\la,0}) = 0,\\
			f_{\la}(x) > 0, \quad x \in (0, x_{\la,0}).
		\end{gather}
	\end{enumerate}
\end{lemma}
\begin{proof}
	It is clear that, by \eqref{eq:fla.def} and Assumption~\ref{asm:W.airy}~\ref{itm:W.even}, $f_{\la}$ is odd.
	
	\ref{itm:fla.unbded} Let $\la > 0$ be arbitrary but fixed. Since, by Assumption~\ref{asm:W.airy}~\ref{itm:W.incr.unbd}, we have $W(x) \to +\infty$ for $x \to +\infty$, it follows that there exists $x_1(\la) > 0$ such that
	\begin{equation*}
		\la - W(x) \le -1, \quad x \ge x_1(\la).
	\end{equation*}
	Hence, for some $C_{\la} \in \R$ and any $x \ge x_1(\la)$, we have
	\begin{equation*}
		\begin{aligned}
			f_{\la}(x) = \int_{0}^{x_1(\la)} (\la - W(t)) \dd t + \int_{x_1(\la)}^{x} (\la - W(t)) \dd t \le C_{\la} + x_1(\la) - x.
		\end{aligned}
	\end{equation*}
	We therefore conclude that \eqref{eq:fla.unbounded} holds.
	 
	\ref{itm:fla.cps} Setting $\la_0$ as per definition \eqref{eq:W.tp.def} or larger, the statement is a direct consequence of \eqref{eq:fla.deriv}, \eqref{eq:W.tp.def} and Assumption~\ref{asm:W.airy}~\ref{itm:W.incr.unbd}.
	
	\ref{itm:fla.zeroes} Let $\la_0$ be as in \ref{itm:fla.cps} and let $\la \ge \la_0$ be arbitrary but fixed. Since $f_{\la}(0) = 0$, $f'_{\la}(x) > 0$ for $x \in (0, x_{\la})$, $f'_{\la}(x) < 0$ for $x \in (x_{\la}, +\infty)$ (recall that $f'_{\la}(x) = - W(x) + \la$)  and moreover, by \eqref{eq:fla.unbounded}, $f_{\la}(x)$ is eventually negative, we deduce that there exists a unique $x_{\la,0} > x_{\la}$ such that $f_{\la}(x_{\la,0}) = 0$ and $f_{\la}(x)$ must be strictly positive in $(0, x_{\la,0})$, as claimed.
\end{proof}

The behaviour of $f_{\la}$ is sketched in Fig.~\ref{fig:fla.graph} (see also Definition~\eqref{eq:omegala.def} and Lemma~\ref{lem:omega.subint}).
\begin{figure}[h]
	\includegraphics[scale=0.95]{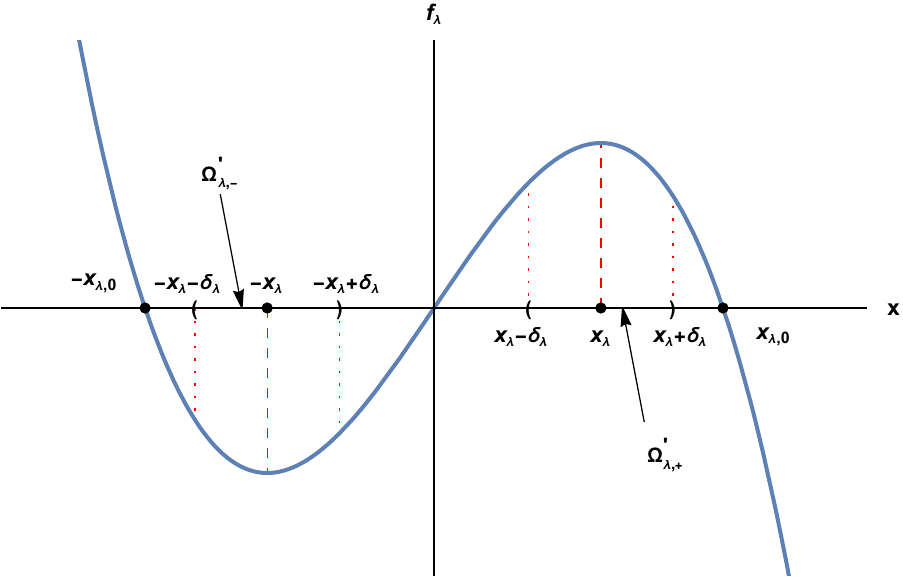}
	\caption{Sketch of $f_{\la}(x)$, defined by \eqref{eq:fla.def}, for sufficiently large $\la$.}
	\label{fig:fla.graph}
\end{figure}

Our main result in this section is the following.

\begin{theorem}
	\label{thm:airy.norm}
	Let $W$ satisfy Assumption~\ref{asm:W.airy}, let $\opA$ be the generalised Airy operator \eqref{eq:airy.op.def} in $\Lt(\R)$ and let $\la > 0$. Then
	\begin{equation}
		\label{eq:A.resnorm}
		\| (\opA - \la)^{-1} \| = \sqrt{\pi} (W'(x_{\la}))^{-\frac12} \exp \left(2 f_{\la}(x_{\la})\right) (1 + o(1)), \quad \la \to +\infty,
	\end{equation}
	where $x_{\la}$ and $f_{\la}$ are defined by \eqref{eq:W.tp.def} and \eqref{eq:fla.def}, respectively.
\end{theorem}

\begin{example}
	\label{ex:W.resolvent}
	In order to illustrate Theorem~\ref{thm:airy.norm}, we re-visit the instances shown in Example~\ref{ex:W.basic}.
	\begin{enumerate}[\upshape (i), wide]
		\item For the slowly-growing function $W(x) = \log \langle x \rangle^{p}$, $p > 0$, it is straightforward to verify that $W'(x) = p x \langle x \rangle^{-2}$, $f_{\la}(x) = (\la + p) x - x W(x) - p \arctan(x)$ and $x_{\la} = \sqrt{\exp(2 \la / p) - 1}$. This yields the estimate
		\begin{equation}
			\label{eq:A.resnorm.ex.log}
			\| (\opA - \la)^{-1} \| = \sqrt{\frac{\pi}{p}} \exp\left(2 p \sqrt{\exp \left(\frac{2\la}{p} \right) - 1} + \frac{\la}{2 p} - p \pi\right) (1 + o(1)), \quad \la \to +\infty.
		\end{equation}
		\item \label{itm:W.poly.resolvent}For $W(x) = |x|^{p}$, $p > 0$, we have $W'(x) = p \sgn(x) |x|^{p - 1}$, $f_{\la}(x) = x (\la - W(x) / (p + 1))$, $x_{\la} = \la^{1/p}$ and the estimate
		\begin{equation}
			\label{eq:A.resnorm.ex.poly}
			\| (\opA - \la)^{-1} \| = \sqrt{\frac{\pi}{p}} \la^{\frac{1 - p}{2 p}} \exp\left(\frac{2 p}{p + 1} \la^{\frac{1 + p}{p}}\right) (1 + o(1)), \quad \la \to +\infty.
		\end{equation}
		We note that, in the particular case $p = 2$ (the complex Airy operator on Fourier space), we recover the estimate found elsewhere (see \eg~\cite[Cor.~1.4]{BordeauxMontrieux-2013})
		\begin{equation*}
			\| (\opA - \la)^{-1} \| = \sqrt{\frac{\pi}{2}} \la^{-\frac14} \exp\left(\frac43 \la^{\frac32}\right) (1 + o(1)), \quad \la \to +\infty.
		\end{equation*}
		\item \label{itm:W.exp.resolvent}For the fast-growing function $W(x) = e^{|x|^{p}}$, $p > 0$, one can similarly check that $W'(x) = p \sgn(x) |x|^{p-1} W(x)$, $f_{\la}(x) = \la x - W(x) F_p(x)$, where $F_p(x)$ denotes the generalised Dawson's integral (see \cite[Sec.~7.16]{DLMF} or \cite[Eq.~1.1]{dijkstra1977continued}), and $x_{\la} = (\log(\la))^{\frac1{p}}$. The estimate of the resolvent norm as $\la \to +\infty$ is
		\begin{equation*}
			\begin{aligned}
				\| (\opA - \la)^{-1} \| &= \sqrt{\frac{\pi}{p}} \la^{-\frac12} (\log(\la))^{\frac{1-p}{p}} \exp\left(2 \la \left((\log(\la))^{\frac1{p}} - F_p\left((\log(\la))^{\frac1{p}}\right)\right)\right)\\
				&\hspace{3.5cm} (1 + o(1))
			\end{aligned}
		\end{equation*}
		(note that $F_p(x) = (1 / p) x^{1-p} (1 + \BigO(x^{-p}))$ for $x \to +\infty$, see \eg~\cite[Secs.~7.16,~8.11]{DLMF} or \cite[Eq.~2.5]{dijkstra1977continued}).
	\end{enumerate}
\end{example}

\subsection{Proof of Theorem~\ref{thm:airy.norm}}
\label{ssec:norm.resolvent.proof}
To prove the theorem, we shall firstly introduce a certain $\Lt(\R)$ integral operator and show that it is bounded, establishing in the process that the right-hand side of \eqref{eq:A.resnorm} is an (asymptotic) upper bound for its norm (Proposition~\ref{prop:airy.norm.ubound}). We will then proceed to prove that the upper bound is in fact optimal (Proposition~\ref{prop:airy.norm.lbound}). Finally, in the theorem's proof proper, we show that our integral operator is indeed the resolvent $(\opA - \la)^{-1}$.

We begin by proving two lemmas that describe some further important properties of $f_{\la}$. Let
\begin{equation}
	\label{eq:omegala.def}
	\Omega'_{\la,\pm} := \left( \pm x_{\la} - \delta_{\la}, \pm x_{\la} + \delta_{\la} \right), \quad \delta_{\la} := \delta x_{\la}^{-\nu}, \quad 0 < \delta < \frac 14,
\end{equation}
where $\delta$ will be determined in Lemmas~\ref{lem:omega.subint} and~\ref{lem:laplace.rplus} and $\nu \ge -1$ (see Assumption~\ref{asm:W.airy}~\ref{itm:W.control.derv}). The above choice for the width of $\Omega'_{\la,\pm}$ implies that $W(x)$ is approximately equal to $W(x_{\la})$ inside those intervals (see Remark~\ref{rmk:asm}~\ref{itm:asm.W.variation.cte}), a fact which we shall repeatedly rely upon in the proofs below.

\begin{lemma}
	\label{lem:omega.subint}
	Let $W$ satisfy Assumption~\ref{asm:W.airy} and let $x_{\la}$, $x_{\la,0}$ and $\Omega'_{\la,+}$ be as defined in \eqref{eq:W.tp.def}, \eqref{eq:fla.zeroes} and \eqref{eq:omegala.def}, respectively. Then as $\la \to +\infty$, we have:
	\begin{enumerate}[\upshape (i)]
		\item \label{itm:omega.approx.cte}
		\begin{equation}
			\label{eq:omega.approx.cte}
			x_{\la} \pm \delta_{\la} \approx x_{\la}.
		\end{equation}
		\item \label{itm:omega.subset.fla.zeroes}
		\begin{equation}
			\label{eq:omega.subset.fla.zeroes}
			\Omega'_{\la,+} \subset (0, x_{\la,0}).
		\end{equation}
	\end{enumerate}
\end{lemma}
\begin{proof}
	\ref{itm:omega.approx.cte} By \eqref{eq:xlapm.unbounded} and $\nu \ge -1$, we find as $\la \to +\infty$
	\begin{equation*}
		x_{\la} \pm \delta_{\la} = x_{\la} (1 \pm \delta x_{\la}^{-1 - \nu}) \approx x_{\la}.
	\end{equation*}

	\ref{itm:omega.subset.fla.zeroes} Since $x_\la - \delta_{\la} > 0$ as $\la \to +\infty$ by \ref{itm:omega.approx.cte}, it is enough to prove $f_{\la}(x_{\la}+\delta_{\la}) > 0$ (see Lemma~\ref{lem:fla.prop}) for $\la \to +\infty$. For $\la_0$ as in the statement of Lemma~\ref{lem:fla.prop} and any $\la \ge \la_0$, we have (with $x_0$ as in Assumption~\ref{asm:W.airy} and applying \eqref{eq:fla.deriv} and integration by parts)
	\begin{equation*}
		\begin{aligned}
			f_{\la}(x_{\la}) &= \int_{0}^{x_{\la}} (\la - W(t)) \dd t = \int_{0}^{x_0} (\la - W(t)) \dd t + \int_{x_0}^{x_{\la}} (\la - W(t)) \dd t\\
			&= \int_{0}^{x_0} (\la - W(t)) \dd t - (\la - W(x_0)) x_0 + \int_{x_0}^{x_{\la}} t W'(t) \dd t\\
			&= \int_{0}^{x_0} (W(x_0) - W(t)) \dd t + \int_{x_0}^{x_{\la}} t W'(t) \dd t.
		\end{aligned}
	\end{equation*}
	The first term in the preceding line is a constant whereas the second one diverges as $\la \to +\infty$ (we may assume $x_0 \ge 1$), hence
	\begin{equation*}
		f_\la(x_{\la}) = \int_{x_0}^{x_{\la}} t W'(t) \dd t \; (1 + o(1)) \gs \int_{x_{\la} - \eps x_{\la}^{-\nu}}^{x_{\la}} t W'(t) \dd t,
	\end{equation*}
	for sufficiently small $\eps > 0$ and $\la \to +\infty$. By \eqref{eq:W.approx.cte}, we deduce that $\min \{W'(t) : t \in [x_{\la} - \eps x_{\la}^{-\nu}, x_{\la}]\} \approx W'(x_{\la})$. Furthermore, $x_{\la} - \eps x_{\la}^{-\nu} \approx x_{\la}$ for small enough $\eps > 0$ by \eqref{eq:omega.approx.cte}. It follows
	\begin{equation}
		\label{eq:fla.xla.lbound}
		f_{\la}(x_{\la}) \gs x_{\la}  W'(x_{\la}) \eps x_{\la}^{-\nu} \approx W'(x_{\la}) x_{\la}^{1 - \nu}, \quad \la \to +\infty.
	\end{equation}
	Therefore by \eqref{eq:fla.xla.lbound} and \eqref{eq:xlapm.unbounded}
	\begin{equation}
		\label{eq:fla.res.1}
		\frac{(\Upsilon_1(x_{\la}))^{-2}}{f_{\la}(x_{\la})} = \frac{W'(x_{\la})}{f_{\la}(x_{\la}) x_{\la}^{2 \nu}} \ls \frac{W'(x_{\la})}{W'(x_{\la}) x_{\la}^{1 + \nu}} = \BigO(1), \quad \la \to +\infty,
	\end{equation}
	and hence (with $\xi_{\la} \in (x_{\la}, x_{\la} + \delta_{\la})$ and applying \eqref{eq:fla.deriv})
	\begin{equation}
		\label{eq:fla.taylor}
		\begin{aligned}
			f_{\la}(x_{\la} + \delta_{\la}) &= f_{\la}(x_{\la}) + f'_{\la}(x_{\la}) \delta_{\la} + \frac12 f''_{\la}(x_{\la}) \delta_{\la}^2 + \frac16 f'''_{\la}(\xi_{\la}) \delta_{\la}^3\\
			&= f_{\la}(x_{\la}) \left(1 - \frac12 \frac{W'(x_{\la})}{f_{\la}(x_{\la})} \delta_{\la}^2 \left(1 + \frac13 \frac{W''(\xi_{\la})}{W'(x_{\la})} \delta_{\la}\right)\right)\\
			&= f_{\la}(x_{\la}) \left(1 - \frac12 \delta^2 \frac{(\Upsilon_1(x_{\la}))^{-2}}{f_{\la}(x_{\la})} \left(1 + \frac13 \frac{W''(\xi_{\la})}{W'(x_{\la})} \delta_{\la}\right)\right).
		\end{aligned}
	\end{equation}
	By \eqref{eq:fla.res.1}, we obtain $\delta^2 (\Upsilon_1(x_{\la}))^{-2} / f_{\la}(x_{\la}) = \BigO(\delta^2)$ and, by Assumption~\ref{asm:W.airy}~\ref{itm:W.control.derv}, \eqref{eq:omega.approx.cte} and \eqref{eq:W.approx.cte}, we deduce $|W''(\xi_{\la)}| \delta_{\la} / W'(x_{\la}) \ls \delta$. Hence, using \eqref{eq:fla.taylor} and choosing $\delta > 0$ sufficiently small in \eqref{eq:omegala.def}, we have
	\begin{equation*}
		f_{\la}(x_{\la} + \delta_{\la}) > f_{\la}(x_{\la}) \left(1 - \BigO(\delta^2)\right) > 0, \quad \la \to +\infty,
	\end{equation*}
	as claimed.
\end{proof}

Our next result is the critical ingredient in the calculation of estimate \eqref{eq:A.resnorm}. It uses the idea of Laplace's method and slightly extends it to provide suitable estimates for the proof of our claim.

\begin{lemma}
	\label{lem:laplace.rplus}
	Let $\la > 0$ and let $x_{\la}, \Omega'_{\la,\pm}, f_{\la}(x)$ be as defined by \eqref{eq:W.tp.def}, \eqref{eq:omegala.def} and \eqref{eq:fla.def}, respectively. Then for any $M > 0$ and with $\la \to +\infty$
	\begin{align}
		\label{eq:laplace.Ilaplus} I_{\la,+} := \int_{\Omega'_{\la,+}} e^{M f_{\la}(x)} \dd x &= \sqrt{\frac{2 \pi}{M}} (W'(x_{\la}))^{-\frac12} e^{M f_{\la}(x_{\la})} (1 + o(1)), \quad\\
		\label{eq:laplace.Ilaminus} I_{\la,-} := \int_{\Omega'_{\la,-}} e^{M f_{\la}(x)} \dd x &= \sqrt{\frac{2 \pi}{M}} (W'(x_{\la}))^{-\frac12} e^{M f_{\la}(x_{\la})} o(1). \quad
	\end{align}
\end{lemma}
\begin{proof}
	For $\la_0 > 0$ as in Lemma~\ref{lem:fla.prop} and any $\la \ge \la_0$, let $\rho := |f''_{\la}(x_{\la})|^{-\frac12} > 0$ and $\Omega'_{\la,\rho} := (-\delta_{\la} \rho^{-1}, \delta_{\la} \rho^{-1})$. For every $y \in \Omega'_{\la,\rho}$, we define
	\begin{equation}
		\label{eq:gla.def}
		g_{\la}(y) := f_{\la}(\rho y + x_{\la}).
	\end{equation}
	Taylor-expanding $g_{\la}$ around $0$ and using the facts that $g'_{\la}(0) = \rho f'_{\la}(x_{\la}) = 0$ (refer to \eqref{eq:W.tp.def} and \eqref{eq:fla.deriv}) and $g''_{\la}(0) = \rho^2 f''_{\la}(x_{\la}) = -1$ (by Assumption~\ref{asm:W.airy}~\ref{itm:W.incr.unbd}, \eqref{eq:fla.deriv} and our choice of $\rho$) we obtain
	\begin{equation}
		\label{eq:gla.taylor}
		g_{\la}(y) = g_{\la}(0) - \frac12 y^2 + \frac16 g'''_{\la}(s y) y^3, \quad y \in \Omega'_{\la,\rho},
	\end{equation}
	where $s = s(y, \la)$ and $0 < s < 1$. Furthermore, by \eqref{eq:fla.deriv} and Assumption~\ref{asm:W.airy}~\ref{itm:W.control.derv}
	\begin{equation*}
		\begin{aligned}
			|g'''_{\la}(s y)| &= |f'''_{\la}(\rho s y + x_{\la})| \rho^3 = \frac{|W''(\rho s y + x_{\la})|}{(W'(x_{\la}))^{\frac32}}\\
			&\ls \frac{W'(\rho s y + x_{\la})}{(W'(x_{\la}))^{\frac32}} (\rho s y + x_{\la})^{\nu}.
		\end{aligned}
	\end{equation*}
	For any $y \in \Omega'_{\la,\rho}$, we have $|\rho s y| \le \frac12 x_{\la}^{-\nu}$ by \eqref{eq:omegala.def} and hence $(x_{\la}^{-1} \rho s y + 1)^{\nu} \approx 1$, \ie~$(\rho s y + x_{\la})^{\nu} \approx x_{\la}^{\nu}$. Combining this fact with \eqref{eq:W.approx.cte}, we deduce
	\begin{equation}
		\label{eq:gla.remainder.estimate.1}
		|g'''_{\la}(s y)| \ls (\rho s y + x_{\la})^{\nu} (W'(x_{\la}))^{-\frac12} \ls \Upsilon_1(x_{\la}), \quad y \in \Omega'_{\la,\rho}.
	\end{equation}
	Observing that $\delta_{\la} \rho^{-1} = (\Upsilon_1(x_{\la}))^{-1} \delta$, we have $|g'''_{\la}(s y) y| \ls \delta$ for any $y \in \Omega'_{\la,\rho}$ and consequently, by choosing a sufficiently small $\delta$ in \eqref{eq:omegala.def}, we get
	\begin{equation}
		\label{eq:gla.remainder.estimate.2}
		1 - \frac13 g'''_{\la}(s y) y \approx 1.
	\end{equation}

	In order to estimate $I_{\la,+}$, we change variable $x \to \rho y + x_{\la}$ and apply \eqref{eq:gla.def} and \eqref{eq:gla.taylor} to derive
	\begin{equation}
		\label{eq:Ilaplus.dct.estimate}
		I_{\la,+} = e^{M g_{\la}(0)} \rho \intR e^{-\frac12 M y^2 (1 - \frac13 g'''_{\la}(s y) y)} \chi_{\Omega'_{\la,\rho}}(y) \dd y.
	\end{equation}
	Note that, by \eqref{eq:gla.remainder.estimate.2}, there exists some $C > 0$ such that
	\begin{equation*}
		e^{-\frac12 M y^2 (1 - \frac13 g'''_{\la}(s y) y)} \chi_{\Omega'_{\la,\rho}}(y) \le e^{-\frac12 C M y^2}, \quad y \in \R, \quad \la \ge \la_0.
	\end{equation*}
	Furthermore, $\Upsilon_1(x_{\la}) = o(1)$ as $\la \to +\infty$ by Assumption~\ref{asm:W.airy}~\ref{itm:W.upsilon} and therefore $\delta_{\la} \rho^{-1} \to +\infty$ for $\la \to +\infty$. Noting also \eqref{eq:gla.remainder.estimate.1}, we have for any fixed $y \in \R$
	\begin{equation*}
		\lim_{\la \to +\infty} e^{-\frac12 M y^2 (1 - \frac13 g'''_{\la}(s y) y)} \chi_{\Omega'_{\la,\rho}}(y) = e^{-\frac12 M y^2}.
	\end{equation*}
	Finally, we apply the dominated convergence theorem to \eqref{eq:Ilaplus.dct.estimate} to deduce
	\begin{equation*}
		I_{\la,+} = e^{M g_{\la}(0)} \rho \sqrt{\frac{2 \pi}{M}} (1 + o(1)), \quad \la \to +\infty,
	\end{equation*}
	which yields \eqref{eq:laplace.Ilaplus}.
	
	To prove \eqref{eq:laplace.Ilaminus}, we change variable in the integrand of $I_{\la,-}$ ($x \to -\rho y - x_{\la}$) to obtain (using the fact that $f_{\la}$ is odd and arguing as above)
	\begin{equation*}
		\begin{aligned}
			I_{\la,-} &= e^{-M g_{\la}(0)} \rho \intR e^{\frac12 M y^2 (1 - \frac13 g'''_{\la}(s y) y)} \chi_{\Omega'_{\la,\rho}}(y) \dd y\\
			&\le e^{-M g_{\la}(0)} \rho \intR e^{\frac12 C' M y^2} \chi_{\Omega'_{\la,\rho}}(y) \dd y = 2 e^{-M g_{\la}(0)} \rho \int_{0}^{\delta_{\la} \rho^{-1}} e^{\frac12 C' M y^2} \dd y,
		\end{aligned}
	\end{equation*}
	with some $C' > 0$. It is straightforward to verify that $e^{-x^2} x \int_{0}^{x} e^{y^2} \dd y = \BigO(1)$ as $x \to +\infty$ and hence
	\begin{equation*}
		\begin{aligned}
			I_{\la,-} &\ls e^{-M g_{\la}(0)} \rho \frac{e^{\frac12 C' M \delta_{\la}^2 \rho^{-2}}}{\delta_{\la} \rho^{-1}}\\
			&= (W'(x_{\la}))^{-\frac12} e^{M f_{\la}(x_{\la})} \left(\delta^{-1} \Upsilon_1(x_{\la}) e^{-2 M f_{\la}(x_{\la}) \left(1 - \frac{C' \delta^2 (\Upsilon_1(x_{\la}))^{-2} }{4 f_{\la}(x_{\la})}\right)}\right),
		\end{aligned}
	\end{equation*}
	for $\la \to +\infty$. Recalling \eqref{eq:fla.res.1} and choosing an adequately small value for $\delta > 0$ in \eqref{eq:omegala.def}, it follows that
	\begin{equation*}
		1 - \delta^2 \frac{C' (\Upsilon_1(x_{\la}))^{-2}}{4 f_{\la}(x_{\la})} > 0,\quad \la \to +\infty.
	\end{equation*}
	We have $M > 0$ by assumption and it was shown in Lemma~\ref{lem:fla.prop}~\ref{itm:fla.zeroes} that $f_\la(x_{\la}) > 0$ for large enough $\la$; furthermore, by Assumption~\ref{asm:W.airy}~\ref{itm:W.upsilon}, $\Upsilon_1(x_{\la}) = o(1)$ as $\la \to +\infty$. Hence we conclude that \eqref{eq:laplace.Ilaminus} holds.
\end{proof}

\subsubsection{The upper bound}
\label{sssec:u.bound.proof}
With $\la > 0$ and $f_{\la}$ as in \eqref{eq:fla.def}, we define the non-negative kernel
\begin{equation}
	\label{eq:kla.def}
	k_{\la}(x,y) := 
	\begin{cases}
		e^{f_{\la}(y) - f_{\la}(x)}, &  x < y,\\[1mm]
		0, &  x \ge y,
	\end{cases}
	\quad x, y \in \R,
\end{equation}
and we let $\opT_{k_{\la}}$ be the associated integral operator as in \eqref{eq:Tk.def}. Furthermore, we introduce the weight functions in $\R$
\begin{align}
	\label{eq:A.schur.p} p_{\la}(x) &= 
	\begin{cases}
		e^{-f_{\la}(x)}, &  x \in [-x_{\la} - \delta_{\la}, x_{\la} + \delta_{\la}],\\[1mm]
		e^{-f_{\la}(x_{\la} + \delta_{\la})}, &  x > x_{\la} + \delta_{\la},\\[1mm]
		e^{-f_{\la}(-x_{\la} - \delta_{\la})}, &  x < -x_{\la} - \delta_{\la},
	\end{cases}
	\\
	\label{eq:A.schur.q} q_{\la}(y) &= 
	\begin{cases}
		e^{f_{\la}(y)}, &  y \in [-x_{\la} - \delta_{\la}, x_{\la} + \delta_{\la}],\\[1mm]
		e^{f_{\la}(x_{\la} + \delta_{\la})}, &  y > x_{\la} + \delta_{\la},\\[1mm]
		e^{f_{\la}(-x_{\la} - \delta_{\la})}, &  y < -x_{\la} - \delta_{\la},
	\end{cases}
\end{align}
where $x_{\la}$ and $\delta_{\la}$ are defined in \eqref{eq:W.tp.def} and \eqref{eq:omegala.def}, respectively. Note that the choice of weights is motivated to ensure that the resulting constants $\alpha, \beta$ in Schur's test (see Sub-section~\ref{ssec:schur.prelim}) are optimal, which cannot be achieved with the trivial weights $p_{\la}(x) = q_{\la}(x) = 1$.

\begin{proposition}
	\label{prop:airy.norm.ubound}
	Let $W$ satisfy Assumption~\ref{asm:W.airy} and let $\la > 0$ and $\opT_{k_{\la}}$ be the integral operator associated with the kernel \eqref{eq:kla.def}. Then $\opT_{k_{\la}}$ is a bounded operator on $\Lt(\R)$ and moreover
	\begin{equation}
		\label{eq:A.resnorm.ubound}
		\| \opT_{k_{\la}} \| \le \sqrt{\pi} (W'(x_{\la}))^{-\frac12}\exp \left(2 f_{\la}(x_{\la})\right) (1 + o(1)), \quad \la \to +\infty,
	\end{equation}
	where $x_{\la}$ and $f_{\la}$ are defined as in \eqref{eq:W.tp.def} and \eqref{eq:fla.def}, respectively.
\end{proposition}
\begin{proof}
	Our task is to estimate the integrals
	\begin{equation*}
		\begin{aligned}
			I_{\la}(x) &:= \intR k_{\la}(x,y) q_{\la}(y) \dd y,\\
			J_{\la}(y) &:= \intR k_{\la}(x,y) p_{\la}(x) \dd x,
		\end{aligned}
		 \quad x, y \in \R, \quad \la \to +\infty,
	\end{equation*}
	with $k_{\la}(x,y)$ as in \eqref{eq:kla.def} and $p_{\la}(x), q_{\la}(y)$ as in \eqref{eq:A.schur.p} and \eqref{eq:A.schur.q}, respectively. In what follows, we shall assume $\la \ge \la_0$ with $\la_0$ as in Lemma~\ref{lem:fla.prop}.
	
	We split $\R$ into subintervals and evaluate $I_{\la}(x)$ in each one in turn.
	
	\begin{enumerate}[\upshape (i), wide]
		\item \label{itm:Ila.1} \underline{$x \in [x_{\la} + \delta_{\la}, +\infty)$}
		\begin{equation*}
			\begin{aligned}
				I_{\la}(x) &= e^{-f_{\la}(x) + f_{\la}(x_{\la} + \delta_{\la})} \int_{x}^{+\infty} e^{f_{\la}(y)} \frac{W(y) - \la}{W(y) - \la} \dd y\\
				&\le \frac{e^{-f_{\la}(x) + f_{\la}(x_{\la} + \delta_{\la})}}{W(x_{\la} + \delta_{\la}) - \la} \left[-e^{f_{\la}(y)}\right]_{x}^{+\infty} = \frac{e^{f_{\la}(x_{\la} + \delta_{\la})}}{W(x_{\la} + \delta_{\la}) - \la}\\
				&= \frac{e^{2 f_{\la}(x_{\la})}}{(W'(x_{\la}))^{\frac12}} \left(\frac{(W'(x_{\la}))^{\frac12}}{W(x_{\la} + \delta_{\la}) - \la} e^{-2 (f_{\la}(x_{\la}) - f_{\la}(x_{\la} + \delta_{\la}))}\right) e^{-f_{\la}(x_{\la} + \delta_{\la})}.
			\end{aligned}
		\end{equation*}
		Applying the mean-value theorem and \eqref{eq:W.approx.cte} to $W$ and noting that $f_{\la}(x_{\la}) > f_{\la}(x_{\la} +\delta_{\la})$, we find
		\begin{equation*}
			\frac{(W'(x_{\la}))^{\frac12}}{W(x_{\la} + \delta_{\la}) - \la} e^{-2 (f_{\la}(x_{\la}) - f_{\la}(x_{\la} + \delta_{\la}))} \ls \frac{(W'(x_{\la}))^{\frac12}}{W'(x_{\la}) \delta_{\la}} = \delta^{-1} \Upsilon_1(x_{\la})
		\end{equation*}
		and hence
		\begin{equation*}
			I_{\la}(x) = \sqrt{\pi} (W'(x_{\la}))^{-\frac12} e^{2 f_{\la}(x_{\la})} o(1) p_{\la}(x),\quad \la \to +\infty.
		\end{equation*}
		\item \label{itm:Ila.2} \underline{$x \in [x_{\la} - \delta_{\la}, x_{\la} + \delta_{\la})$}
		\begin{equation*}
			I_{\la}(x) = e^{-f_{\la}(x)} \left(\int_{x}^{x_{\la} + \delta_{\la}} e^{2 f_{\la}(y)} \dd y + \int_{x_{\la} + \delta_{\la}}^{+\infty} e^{f_{\la}(y)} q_{\la}(y) \dd y\right).
		\end{equation*}
		Applying \eqref{eq:laplace.Ilaplus} with $M = 2$, we have for $\la \to +\infty$
		\begin{equation*}
			\int_{x}^{x_{\la} + \delta_{\la}} e^{2 f_{\la}(y)} \dd y \le \int_{\Omega'_{\la,+}} e^{2 f_{\la}(y)} \dd y = \sqrt{\pi} (W'(x_{\la}))^{-\frac12} e^{2 f_{\la}(x_{\la})} (1 + o(1)).
		\end{equation*}
		Moreover using \ref{itm:Ila.1}
		\begin{equation*}
			\begin{aligned}
				\int_{x_{\la} + \delta_{\la}}^{+\infty} e^{f_{\la}(y)} q_{\la}(y) \dd y &= e^{f_{\la}(x_{\la} + \delta_{\la})} I_{\la}(x_{\la} + \delta_{\la})\\
				&= \sqrt{\pi} (W'(x_{\la}))^{-\frac12} e^{2 f_{\la}(x_{\la})} o(1),
			\end{aligned}
		\end{equation*}
		as $\la \to +\infty$. Therefore
		\begin{equation*}
			I_{\la}(x) \le \sqrt{\pi} (W'(x_{\la}))^{-\frac12} e^{2 f_{\la}(x_{\la})} (1 + o(1)) p_{\la}(x), \quad \la \to +\infty.
		\end{equation*}
		\item \label{itm:Ila.3} \underline{$x \in [0, x_{\la} - \delta_{\la})$}
		\begin{equation*}
			I_{\la}(x) = e^{-f_{\la}(x)} \left(\int_{x}^{x_{\la} - \delta_{\la}} e^{2 f_{\la}(y)} \dd y + \int_{x_{\la} - \delta_{\la}}^{+\infty} e^{f_{\la}(y)} q_{\la}(y) \dd y\right).
		\end{equation*}
		The first of the above two integrals can be estimated as in \ref{itm:Ila.1} (noting that $f_{\la}(x_{\la}) > f_{\la}(x_{\la} - \delta_{\la})$)
		\begin{equation*}
			\begin{aligned}
				\int_{x}^{x_{\la} - \delta_{\la}} e^{2 f_{\la}(y)} \dd y &\le \frac{e^{2 f_{\la}(x_{\la} - \delta_{\la})} - e^{2 f_{\la}(x)}}{2(\la - W(x_{\la} - \delta_{\la}))} \le \frac{e^{2 f_{\la}(x_{\la} - \delta_{\la})}}{2(\la - W(x_{\la} - \delta_{\la}))}\\
				&\ls (W'(x_{\la}))^{-\frac12} e^{2 f_{\la}(x_{\la})} \delta^{-1} \Upsilon_1(x_{\la}).
			\end{aligned}
		\end{equation*}
		Moreover using \ref{itm:Ila.2} as $\la \to +\infty$
		\begin{equation*}
			\begin{aligned}
				\int_{x_{\la} - \delta_{\la}}^{+\infty} e^{f_{\la}(y)} q_{\la}(y) \dd y &= e^{f_{\la}(x_{\la} - \delta_{\la})} I_{\la}(x_{\la} - \delta_{\la})\\
				&\le \sqrt{\pi} (W'(x_{\la}))^{-\frac12} e^{2 f_{\la}(x_{\la})} (1 + o(1))
			\end{aligned}
		\end{equation*}
		and therefore
		\begin{equation*}
			I_{\la}(x) \le \sqrt{\pi} (W'(x_{\la}))^{-\frac12} e^{2 f_{\la}(x_{\la})} (1 + o(1)) p_{\la}(x), \quad \la \to +\infty.
		\end{equation*}
		\item \label{itm:Ila.4} \underline{$x \in [-x_{\la} + \delta_{\la}, 0)$}
		\begin{equation*}
			I_{\la}(x) = e^{-f_{\la}(x)} \left(\int_{x}^{0} e^{2 f_{\la}(y)} \dd y + \int_{0}^{+\infty} e^{f_{\la}(y)} q_{\la}(y) \dd y\right).
		\end{equation*}
		Proceeding as in \ref{itm:Ila.3} (recall that $f_{\la}(x) < 0$ for $x \in (-x_{\la,0}, 0) \supset \Omega'_{\la,-}$ and $W$ is even)
		\begin{equation*}
			\int_{x}^{0} e^{2 f_{\la}(y)} \dd y \le \frac1{2(\la - W(x_{\la} - \delta_{\la}))} \ls (W'(x_{\la}))^{-\frac12} e^{2 f_{\la}(x_{\la})} \delta^{-1} \Upsilon_1(x_{\la}).
		\end{equation*}
		Moreover using \ref{itm:Ila.3} as $\la \to +\infty$
		\begin{equation*}
			\begin{aligned}
				\int_{0}^{+\infty} e^{f_{\la}(y)} q_{\la}(y) \dd y &= I_{\la}(0) \le \sqrt{\pi} (W'(x_{\la}))^{-\frac12} e^{2 f_{\la}(x_{\la})} (1 + o(1))
			\end{aligned}
		\end{equation*}
		and therefore
		\begin{equation*}
			I_{\la}(x) \le \sqrt{\pi} (W'(x_{\la}))^{-\frac12} e^{2 f_{\la}(x_{\la})} (1 + o(1)) p_{\la}(x), \quad \la \to +\infty.
		\end{equation*}
		\item \label{itm:Ila.5} \underline{$x \in [-x_{\la} - \delta_{\la}, -x_{\la} + \delta_{\la})$}
		\begin{equation*}
			I_{\la}(x) = e^{-f_{\la}(x)} \left(\int_{x}^{-x_{\la} + \delta_{\la}} e^{2 f_{\la}(y)} \dd y + \int_{-x_{\la} + \delta_{\la}}^{+\infty} e^{f_{\la}(y)} q_{\la}(y) \dd y\right).
		\end{equation*}
		Applyinq \eqref{eq:laplace.Ilaminus} with $M = 2$, we have for $\la \to +\infty$
		\begin{equation*}
			\int_{x}^{-x_{\la} + \delta_{\la}} e^{2 f_{\la}(y)} \dd y \le \int_{\Omega'_{\la,-}} e^{2 f_{\la}(y)} \dd y = \sqrt{\pi} (W'(x_{\la}))^{-\frac12} e^{2 f_{\la}(x_{\la})} o(1).
		\end{equation*}
		Moreover using \ref{itm:Ila.4}
		\begin{equation*}
			\begin{aligned}
				\int_{-x_{\la} + \delta_{\la}}^{+\infty} e^{f_{\la}(y)} q_{\la}(y) \dd y &= e^{f_{\la}(-x_{\la} + \delta_{\la})} I_{\la}(-x_{\la} + \delta_{\la})\\
				&\le \sqrt{\pi} (W'(x_{\la}))^{-\frac12} e^{2 f_{\la}(x_{\la})} (1 + o(1)),
			\end{aligned}
		\end{equation*}
		as $\la \to +\infty$. Therefore
		\begin{equation*}
			I_{\la}(x) \le \sqrt{\pi} (W'(x_{\la}))^{-\frac12} e^{2 f_{\la}(x_{\la})} (1 + o(1)) p_{\la}(x), \quad \la \to +\infty.
		\end{equation*}
		\item \label{itm:Ila.6} \underline{$x \in (-\infty, -x_{\la} - \delta_{\la})$}
		\begin{equation*}
			I_{\la}(x) = e^{-f_{\la}(x)} \left(\int_{x}^{-x_{\la} - \delta_{\la}} e^{f_{\la}(y) + f_{\la}(-x_{\la} - \delta_{\la})} \dd y + \int_{-x_{\la} - \delta_{\la}}^{+\infty} e^{f_{\la}(y)} q_{\la}(y) \dd y\right).
		\end{equation*}
		Arguing as in \ref{itm:Ila.1} and using the facts that $W$ is even, $f_{\la}$ is odd and $f_{\la}(x) \ge 0$ in $[x_{\la}, x_{\la} + \delta_{\la}]$, we find
		\begin{equation*}
			\begin{aligned}
				\int_{x}^{-x_{\la} - \delta_{\la}} e^{f_{\la}(y) - f_{\la}(x_{\la} + \delta_{\la})} \dd y &\le \frac{e^{f_{\la}(x)}}{W(x_{\la} + \delta_{\la}) - \la}\\
				&\ls e^{f_{\la}(x)} \frac{\delta^{-1} \Upsilon_1(x_{\la})}{(W'(x_{\la}))^{\frac12}} e^{2 f_{\la}(x_{\la}) + f_{\la}(x_{\la} + \delta_{\la})}.
			\end{aligned}
		\end{equation*}
		Moreover using \ref{itm:Ila.5}
		\begin{equation*}
			\begin{aligned}
				\int_{-x_{\la} - \delta_{\la}}^{+\infty} e^{f_{\la}(y)} q_{\la}(y) \dd y &= e^{-f_{\la}(x_{\la} + \delta_{\la})} I_{\la}(-x_{\la} - \delta_{\la})\\
				&\le \sqrt{\pi} (W'(x_{\la}))^{-\frac12} e^{2 f_{\la}(x_{\la})} (1 + o(1)),
			\end{aligned}
		\end{equation*}
		as $\la \to +\infty$. Therefore (note that $f_{\la}(x) > f_{\la}(-x_{\la} - \delta_{\la})$)
		\begin{equation*}
			I_{\la}(x) \le \sqrt{\pi} (W'(x_{\la}))^{-\frac12} e^{2 f_{\la}(x_{\la})} (1 + o(1)) p_{\la}(x), \quad \la \to +\infty.
		\end{equation*}
	\end{enumerate}
	
	We have therefore shown that for any $x \in \R$
	\begin{equation}
		\label{eq:Ila.ubound}
		I_{\la}(x) \le \sqrt{\pi} (W'(x_{\la}))^{-\frac12} e^{2 f_{\la}(x_{\la})} (1 + o(1)) p_{\la}(x), \quad \la \to +\infty.
	\end{equation}
	We can also prove that for any $y \in \R$
	\begin{equation}
		\label{eq:Jla.ubound}
		J_{\la}(y) \le \sqrt{\pi} (W'(x_{\la}))^{-\frac12} e^{2 f_{\la}(x_{\la})} (1 + o(1)) q_{\la}(y), \quad \la \to +\infty,
	\end{equation}
	by either repeating the arguments used for $I_{\la}$ or simply noting that, for any $y \in \R$, we have $q_{\la}(y) = p_{\la}(-y)$ and $J_{\la}(y) = I_{\la}(-y)$, and applying \eqref{eq:Ila.ubound}.
	
	Finally, from \eqref{eq:Ila.ubound}-\eqref{eq:Jla.ubound}, an appeal to Schur's test (refer to \eqref{eq:schur.assm.2}-\eqref{eq:schur.ubound}) yields \eqref{eq:A.resnorm.ubound}.	
\end{proof}

\subsubsection{The lower bound}
\label{sssec:l.bound.proof}
Our next result shows that \eqref{eq:A.resnorm.ubound} is optimal.

\begin{proposition}
	\label{prop:airy.norm.lbound}
	Let $W$ satisfy Assumption~\ref{asm:W.airy}, and let $\la > 0$ and $\opT_{k_{\la}}$ be the integral operator associated with the kernel \eqref{eq:kla.def}. Then
	\begin{equation}
		\label{eq:A.resnorm.lbound}
		\| \opT_{k_{\la}} \| \ge \sqrt{\pi} (W'(x_{\la}))^{-\frac12} \exp \left(2 f_{\la}(x_{\la})\right) (1 + o(1)), \quad \la \to +\infty,
	\end{equation}
	where $x_{\la}$ and $f_{\la}$ are defined as in \eqref{eq:W.tp.def} and \eqref{eq:fla.def}, respectively.
\end{proposition}
\begin{proof}
	For sufficiently large $\la > 0$ (see Lemma~\ref{lem:fla.prop}) and with $\Omega'_{\la,+}$ as defined in \eqref{eq:omegala.def}, let
	\begin{equation}
		\label{eq:vla.def}
		v_{\la}(y) := 
		\begin{cases}
			e^{f_{\la}(y)}, &  y \in \Omega'_{\la,+},\\[1mm]
			0, &  \text{otherwise}.
		\end{cases}
	\end{equation}
	Then by \eqref{eq:laplace.Ilaplus} with $M = 2$
	\begin{equation}
		\label{eq:vla.norm}
		\| v_{\la} \|^2 = \int_{\Omega'_{\la,+}} e^{2 f_{\la}(y)} \dd y = \sqrt{\pi} (W'(x_{\la}))^{-\frac12} e^{2 f_{\la}(x_{\la})} (1 + o(1)), \quad \la \to +\infty.
	\end{equation}

	Let for any $x \in \R$ and sufficiently large $\la > 0$
	\begin{equation*}
	u_{\la}(x) := \opT_{k_{\la}} v_{\la}(x) = \intR k_{\la}(x,y) v_{\la}(y) \dd y = e^{-f_{\la}(x)} \int_{x}^{+\infty} e^{2 f_{\la}(y)} \chi_{\Omega'_{\la,+}}(y) \dd y,
	\end{equation*}
	where for the last equality we have used the definitions \eqref{eq:kla.def} and \eqref{eq:vla.def}. Then, for any $x \in \Omega'_{\la,-}$, we have $u_{\la}(x) = \exp(-f_{\la}(x)) \| v_{\la} \|^2$ and, using the fact that $f_{\la}$ is odd and \eqref{eq:vla.norm}, we find
	\begin{equation*}
		\| u_{\la} \|^2 \ge \int_{\Omega'_{\la,-}} |u_{\la}(x)|^2 \dd x = \| v_{\la} \|^4 \int_{\Omega'_{\la,-}} e^{-2 f_{\la}(x)} \dd x = \| v_{\la} \|^6.
	\end{equation*}
	Hence by \eqref{eq:vla.norm}
	\begin{equation*}
		\frac{\| u_{\la} \|}{\| v_{\la} \|} \ge \| v_{\la} \|^2 = \sqrt{\pi} (W'(x_{\la}))^{-\frac12} e^{2 f_{\la}(x_{\la})} (1 + o(1)), \quad \la \to +\infty,
	\end{equation*}
	and \eqref{eq:A.resnorm.lbound} follows.
\end{proof}

\subsubsection{The theorem's proof}
\label{sssec:airy.thm.proof}

\begin{proof}[Proof of Theorem~\ref{thm:airy.norm}]
	With $\opT_{k_{\la}}$ determined by the kernel \eqref{eq:kla.def} and appealing to Propositions~\ref{prop:airy.norm.ubound}, \ref{prop:airy.norm.lbound}, it is enough to prove that the equality $(\opA - \la)^{-1} = \opT_{k_{\la}}$ holds on $\CcR$, a dense subset of $\Lt(\R)$, for $\la \to +\infty$.
	
	Let $v \in \CcR$; since, for any fixed and sufficiently large $\la > 0$, we have $\opT_{k_{\la}} \in \sL(\Lt(\R))$, the function $u_{\la}(x) := \opT_{k_{\la}} v(x)$, $x \in \R$, belongs to $\Lt(\R)$ and moreover, with arbitrary $f \in \CcR$, elementary calculations show that
	\begin{equation*}
		\langle \Nt u_{\la} + W u_{\la}, f \rangle = \langle \la u_{\la} + v, f \rangle.
	\end{equation*}
	It follows that $\Nt u_{\la} + W u_{\la} \in \Lt(\R)$, \ie~$u_{\la} \in \Dom(\opA)$, and $\opA u_{\la} = \la u_{\la} + v$.
	
	Furthermore, by \eqref{eq:fla.deriv}, we have for any $x \in \R$ and $\la > 0$
	\begin{equation*}
		\begin{aligned}
			(\opT_{k_{\la}} (\opA - \la) v)(x) &= \int_{x}^{+\infty} e^{f_{\la}(y) - f_{\la}(x)} (-v'(y) + (W(y) - \la)v(y)) \dd y\\
			&= - e^{-f_{\la}(x)} \left[e^{f_{\la}(y)} v(y)\right]_{x}^{+\infty} = v(x),
		\end{aligned}
	\end{equation*}
	which concludes the proof.
\end{proof}

\section{The norm of the semigroup}
\label{sec:sg.norm}
In this section, we provide an estimate of the norm of the semigroup generated by $-\opA$. We refer the reader to \cite[Ex.~6.1.20, Ex.~10.2.9]{Davies-2007} for instances of semigroups that follow formula \eqref{eq:A.sg.formula} (although with different underlying spaces: $C_0(\R)$ and $\Lt(\Rplus)$, respectively).

\begin{theorem}
	\label{thm:airy.sg.norm}
	Let $W$ satisfy Assumption~\ref{asm:W.airy} and let $\opA$ be the generalised Airy operator \eqref{eq:airy.op.def} in $\Lt(\R)$. Then $-\opA$ generates a $C_0$-semigroup of contractions $(\opS_t)_{t \ge 0}$ on $\Lt(\R)$. Moreover, for any $t \ge 0$ and any $f \in \Lt(\R)$, we have
	\begin{equation}
		\label{eq:A.sg.formula}
		(\opS_t f)(x) = \exp (g_t(x)) f(x+t), \quad g_t(x) := -\int_{x}^{x+t} W(s) \dd s, \quad x \in \R,
	\end{equation}
	and there exists $t_0 \ge 0$ such that
	\begin{equation}
		\label{eq:A.sg.norm}
		\| \opS_t \| = \exp \left( -2 \int_{0}^{\frac{t}{2}} W(s) \dd s \right), \quad t \ge t_0.
	\end{equation}
\end{theorem}
\begin{proof}
	Since, by Proposition~\ref{prop:airy.basic.properties}, $-\opA$ is densely-defined, m-dissipative and $\rho(-\opA) \cap (0, +\infty) \ne \emptyset$, the first part of the claim follows immediately from the Lumer-Phillips theorem (see \eg~\cite[Thm.~II.3.15]{Engel-Nagel-book}). Moreover, for any $u_0 \in \Dom(\opA)$, the function $u(\cdot) := \opS_{\cdot} u_0$ solves the homogenous Cauchy problem
	\begin{equation}
		\label{eq:airy.cauchy.problem}
		\left\{
		\begin{aligned}
			u'(t) &= -\opA u(t),\quad t \ge 0,\\
			u(0) &= u_0,
		\end{aligned}
		\right.
	\end{equation}
	(see \cite[Lem.~II.1.3]{Engel-Nagel-book}).
	
	Observing that the function $\Psi(t,x) = \exp(g_t(x)) \Psi_0(x+t)$, with $g_t(x)$ as in \eqref{eq:A.sg.formula}, solves the partial differential equation $\Ntime\Psi(t,x) = \Ntp\Psi(t,x) - W(x) \Psi(t,x)$ with initial condition $\Psi_0 \in \WotR \cap \Dom(W)$, it can be readily shown that, for every $t \ge 0$, the operator $\widetilde\opS_t$ defined for any $f \in \Lt(\R)$ by
	\begin{equation*}
		(\widetilde\opS_t f)(x) := \exp (g_t(x)) f(x+t), \quad x \in \R,
	\end{equation*}
	with $g_t$ as in \eqref{eq:A.sg.formula}, is bounded with $\| \widetilde\opS_t\| \le 1$ and it satisfies the semigroup property (see (i)-(ii) in \cite[p.~167]{Davies-2007}). Furthermore it can be readily verified that the function $t \mapsto \widetilde\opS_t f$ is continuous at $t = 0$ for each $f \in \Lt(\R)$ and hence $(\widetilde\opS_t)_{t \ge 0}$ is a contraction semigroup. Our next step is to show that the generator of $(\widetilde\opS_t)_{t \ge 0}$, which we shall denote by $(\opZ, \Dom(\opZ))$, is in fact  $-\opA$. We claim that it is enough to prove $\opZ \subset -\opA$. Note that, since $(\widetilde\opS_t)_{t \ge 0}$ is a contraction semigroup, it follows that $1 \in \rho(\opZ)$ (see \cite[Thm.~II.1.10~(ii)]{Engel-Nagel-book}) and hence $1 \in \rho(\opZ) \cap \rho(-\opA) \ne \emptyset$ (recall that $\sigma(-\opA) = \emptyset$ by Proposition~\ref{prop:airy.basic.properties}). Therefore $\opZ \subset -\opA$ \textit{implies} $\opZ = -\opA$ (see \cite[Ex.~IV.1.21~(5)]{Engel-Nagel-book}).
	
	Let us take $f \in \Dom(\opZ)$, with $g := \opZ f \in \Lt(\R)$, and let $\phi \in \CcR$ be a test function. Then
	\begin{equation}
		\label{eq:A.sg.gen.weak.1}
		\langle g, \phi \rangle = \lim_{t \to 0^{+}} \langle t^{-1} (\widetilde\opS_t f - f), \phi \rangle = \lim_{t \to 0^{+}} \langle f, t^{-1} (\widetilde\opS_t^* \phi - \phi) \rangle,
	\end{equation}	
	where $(\widetilde\opS_t^*)_{t \ge 0}$ denotes the adjoint semigroup. It is not difficult to verify that its action on $\Lt(\R)$ is given by
	\begin{equation*}
		(\widetilde\opS_t^* f)(x) = \exp(g_t(-x)) f(x - t), \quad x \in \R.
	\end{equation*}
	Assuming that $\supp(\phi) \subset [a, b]$ (with $a, b \in \R$, $a < b$) and $0 < t < \delta$ (with arbitrarily small $\delta > 0$), we have
	\begin{equation}
		\label{eq:A.sg.gen.weak.2}
		\begin{aligned}
			\langle f, t^{-1} (\widetilde\opS_t^* \phi - \phi) \rangle &= \int_{a}^{b} f(x) \exp(g_t(-x)) \frac{\overline{\phi(x - t) - \phi(x)}}{t} \dd x\\
			&\quad + \int_{a}^{b} f(x) \frac{\exp(g_t(-x)) - 1}{t} \overline{\phi(x)} \dd x.
		\end{aligned}
	\end{equation}
	Noting that $g_t(\cdot)$ is continuous and $\phi(\cdot)$ is continuously differentiable, an application of the dominated convergence theorem yields
	\begin{equation*}
		\lim_{t \to 0^{+}} \int_{a}^{b} f(x) \exp(g_t(-x)) \frac{\overline{\phi(x - t) - \phi(x)}}{t} \dd x = -\int_{a}^{b} f(x) \overline{\phi'(x)} \dd x = \langle f, -\phi' \rangle.
	\end{equation*}
	Observe furthermore that, applying the fundamental theorem of calculus for integrable functions (see \eg~\cite[Sec.~11.6]{Titchmarsh1939theory}) and the dominated convergence theorem to the second term in the right-hand side of \eqref{eq:A.sg.gen.weak.2}, we find
	\begin{equation*}
		\lim_{t \to 0^{+}} \int_{a}^{b} f(x) \frac{\exp(g_t(-x)) - 1}{t} \overline{\phi(x)} \dd x = -\int_{a}^{b} f(x) W(x) \overline{\phi(x)} \dd x = \langle f, -W \phi \rangle.
	\end{equation*}
	Returning to \eqref{eq:A.sg.gen.weak.1} with these results, we have shown
	\begin{equation*}
		\langle g, \phi \rangle = \langle f, -\phi' - W \phi \rangle,
	\end{equation*}
	with arbitrary $\phi \in \CcR$, which proves that $f \in \Dom(-\opA)$ and $\opZ f = g = -\opA f$, as claimed. This concludes the proof of the statement that $-\opA$ generates $(\widetilde\opS_t)_{t \ge 0}$. Since a $C_0$-semigroup is uniquely determined by its generator (see \cite[Thm.~6.1.16]{Davies-2007}), it follows that $\widetilde\opS_t = \opS_t$ for all $t \ge 0$.
	
	It remains to determine $\| \opS_t \|$. Note that, for every $t \ge 0$, we have that $g_t$ is differentiable a.e. in $\R$ and
	\begin{equation*}
		g_t'(x) = W(x) - W(x+t).
	\end{equation*}
	By Assumption~\ref{asm:W.airy}~\ref{itm:W.incr.unbd}, we can find $x_1 > x_0$ such that $W(x_1) > \| W \chi_{[-x_0, x_0]} \|_{\infty}$. Choosing $t_0 \ge 2 x_1$ and recalling Assumption~\ref{asm:W.airy}~\ref{itm:W.even}, it is clear that, for any (fixed) $t \ge t_0$, the value $x = -t/2$ solves the equation $g_t'(x) = 0$. We shall show that this solution is unique (the idea of the proof may be best illustrated with a picture: see Fig.~\ref{fig:W_Wtrans.graph}). For this purpose, assume that there exists $\hat{x} \in \R$ such that $g_t'(\hat{x}) = 0$. If $\hat{x} \in [0, x_0]$, then
	\begin{equation*}
		W(\hat{x}) \le \| W \chi_{[-x_0, x_0]} \|_{\infty} < W(x_1) < W(t) \le W(\hat{x} + t)
	\end{equation*}
	which contradicts the fact that $\hat{x}$ is a solution. If $\hat{x} \in (x_0, +\infty)$, it is not possible to have $W(\hat{x}) = W(\hat{x} + t)$ since $W(x)$ is strictly increasing in $(x_0, +\infty)$ and $t > 0$. It therefore follows that $\hat{x} \notin [0, +\infty)$. If $\hat{x} \le -t \le -2 x_1 < -2 x_0$ and $\hat{x} + t < -x_0$, then $\hat{x}$ cannot be a solution because $W(x)$ is strictly decreasing in $(-\infty, -x_0)$; on the other hand, if $\hat{x} \le -t \le -2 x_1 < -x_1$ and $\hat{x} + t \in [-x_0, 0]$, we have
	\begin{equation*}
		W(\hat{x}) > \| W \chi_{[-x_0, x_0]} \|_{\infty} \ge W(\hat{x} + t)
	\end{equation*}
	and therefore $\hat{x}$ cannot be a solution in this case either. This shows that $\hat{x} \notin (-\infty, -t]$ which leaves us with $(-t, 0)$ as the only possible solution domain. Assuming that $\hat{x} \in (-t, -t/2)$ and using Assumptions~\ref{asm:W.airy}~\ref{itm:W.even},~\ref{itm:W.incr.unbd}, we have $W(\hat{x}) > W(t/2)$ and $W(\hat{x} + t) < W(t/2)$, which once again excludes $\hat{x}$ being a solution. Finally, if $\hat{x} \in (-t/2, 0)$, then $W(t/2) > W(\hat{x})$ and $W(\hat{x} + t) > W(t/2)$, contradicting the fact that $\hat{x}$ is a solution and completing the proof of the uniqueness of $x = -t/2$ as a solution to $g_t'(x) = 0$. Since $\lim_{|x| \to +\infty} W(x) = +\infty$ by Assumption~\ref{asm:W.airy}~\ref{itm:W.even}-\ref{itm:W.incr.unbd}, it follows that $\lim_{|x| \to +\infty} g_t(x) = -\infty$ for any fixed $t \ge t_0$. Consequently we can find $x_2(t) > t/2$ such that $g_t(x) < g_t(-t/2)$ for all $|x| \ge x_2(t)$. But $g_t$ is continuous and therefore $x = -t/2$ is its maximum in $[-x_2(t), x_2(t)]$ and indeed in $\R$. Hence $\sup_{x \in \R} g_t(x) = g_t(-t/2)$ for every $t \ge t_0$ and we conclude that $\| \opS_t \| = \exp (g_t(-t/2))$, which proves \eqref{eq:A.sg.norm}.
	\begin{figure}[h]
		\includegraphics[scale=0.95]{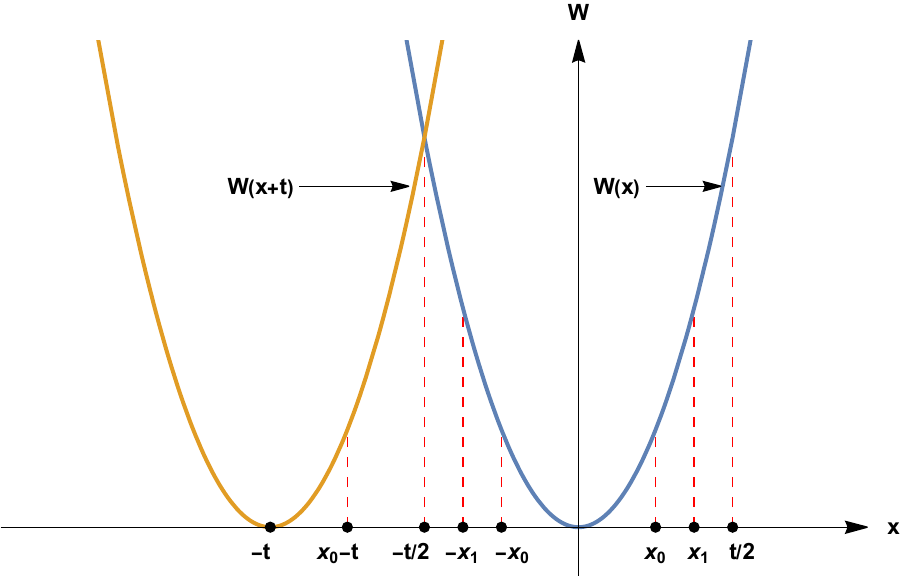}
		\caption{Sketch of the solution to $g_t'(x) = W(x) - W(x+t) = 0$ with fixed $t \ge t_0$.}
		\label{fig:W_Wtrans.graph}
	\end{figure}
\end{proof}

\begin{remark}
	If $W$ is sufficiently regular around $0$, \eg~$W \in C^2(\R)$ and $W'(x) > 0$, $x > 0$, then we can take $t_0 = 0$ in the statement of Theorem~\ref{thm:airy.sg.norm}.
\end{remark}

\begin{example}
	\label{ex:W.sg}
	We estimate the semigroup norm for the instances in Example~\ref{ex:W.basic} (refer also to Example~\ref{ex:W.resolvent}).
	\begin{enumerate}[\upshape (i), wide]
		\item For $W(x) = \log \langle x \rangle^{p}$, $p > 0$, we find $\int_{0}^{t/2} W(s) \dd s = W(t/2) t / 2 - p t / 2 + p \arctan(t/2)$ and hence
		\begin{equation*}
			\| \opS_t \| = \langle t/2 \rangle^{- p t} \exp \left(p (t - \pi)\right) (1 + o(1)), \quad t \to +\infty.
		\end{equation*}
		\item For $W(x) = |x|^{p}$, $p > 0$, we have $\int_{0}^{t/2} W(s) \dd s = (2^{-(p + 1)} / (p + 1)) t^{p + 1}$ and hence
		\begin{equation*}
			\| \opS_t \| = \exp \left(-\frac{2^{-p}}{p + 1} t^{p + 1}\right), \quad t \ge 0.
		\end{equation*}
		In the particular case of the complex Airy operator (\ie~$p = 2$), this yields $\| \opS_t \| = \exp(-t^{3} / 12)$, as expected (see \cite[Eq.~14.3.9]{Helffer-2013-book}).
		\item For the fast-growing function $W(x) = e^{|x|^{p}}$, $p > 0$, we have $\int_{0}^{t/2} W(s) \dd s = e^{(t/2)^{p}} F_p(t/2)$, where $F_p(x)$ represents the generalised Dawson's integral (see Example~\ref{ex:W.resolvent}~\ref{itm:W.exp.resolvent}). This results in the estimate
		\begin{equation*}
			\| \opS_t \| = \exp \left(-\frac{2^{p}}{p} t^{1-p}  e^{(\frac{t}{2})^{p}} \left(1 + o(1)\right)\right), \quad t \to +\infty.
		\end{equation*}
	\end{enumerate}
\end{example}

\section{Further remarks}
\label{sec:further.remarks}

\subsection{The level curves of the resolvent}
\label{ssec:level.curves}
To illustrate the role that generalised Airy operators play in the study of more complex non-self-adjoint operators, we show the level curves for the norm of the resolvent of some particular operators analysed in detail elsewhere (\cite{ArSi-resolvent-2022}, \cite{arnal2022resolvent-dwe}).

\subsubsection{Schr\"odinger operators with complex potentials}
\label{sssec:schrodinger.level.curves}
Let us assume that $i V$ satisfies \cite[Asm.~4.1]{ArSi-resolvent-2022}; in particular, $V$ is non-negative, sufficiently regular, even and regularly varying with index $p > 0$, \ie
\begin{equation*}
	\exists p > 0, \quad \forall x > 0, \quad \lim_{t \to +\infty} \frac{V(t x)}{V(t)} = |x|^{p}.
\end{equation*}

Let
\begin{equation*}
	\opH = \Dt + iV,
	\quad \Dom(H) = W^{2,2}(\R) \cap \Dom(V),
\end{equation*}
be the corresponding Schr\"odinger operator. Assume $a > 0$ to be large and define the positive real numbers $t_a$ via the equation
\begin{equation*}
	t_a V(t_a) = 2 \sqrt{a}.
\end{equation*}
In \cite[Prop.~5.2]{ArSi-resolvent-2022}, it has been shown that, for curves adjacent to the real axis inside the numerical range of $\opH$
\begin{equation*}
	\la_a := a + i b(a),
\end{equation*}
with $b: \Rplus \to \Rplus$ satisfying suitable conditions (see \cite[Sub-sec.~5.1.2]{ArSi-resolvent-2022}), we have
\begin{equation}
	\label{eq:schrodinger.res.norm.gral}
	\| (\opH - \la_a)^{-1} \| = \| (\opA_p - \mu_a)^{-1} \| (V(t_a))^{-1} (1 + o(1)), \quad a \to +\infty,
\end{equation}
with
\begin{equation*}
	\opA_p = \Nt + |x|^p, \quad 
	\Dom(\opA_p) = W^{1,2}(\R) \cap \Dom(|x|^p),
\end{equation*}
and $\mu_a := b(a) (V(t_a))^{-1}$. Assuming that $\mu_a \to +\infty$ as $a \to +\infty$ (\ie~$\la_a$ lies inside the numerical range but outside the critical region \cite[Eq.~5.3]{ArSi-resolvent-2022}) and appealing to Example~\ref{ex:W.resolvent}~\ref{itm:W.poly.resolvent}, we have
\begin{equation}
	\label{eq:schrodinger.ga.res.norm}
	\| (\opA_p - \mu_a)^{-1} \| = \sqrt{\frac{\pi}{p}} \mu_a^{\frac{1 - p}{2 p}} \exp\left(\frac{2 p}{p + 1} \mu_a^{\frac{p + 1}{p}}\right) (1 + o(1)), \quad a \to +\infty.
\end{equation}
Combining \eqref{eq:schrodinger.res.norm.gral}, \eqref{eq:schrodinger.ga.res.norm} and substituting $\| (\opH - \la_a)^{-1} \| = \eps^{-1}$, with $\eps > 0$, we obtain (refer to \cite[Sub-sec.~5.1.2]{ArSi-resolvent-2022} for details)
\begin{equation*}
	b(a) = \left(\frac{p + 1}{2p}\right)^{\frac{p}{p + 1}} V(t_a) \left(\log(V(t_a) \eps^{-1})\right)^{\frac{p}{p + 1}} (1 + o(1)), \quad a \to +\infty.
\end{equation*}
We conclude by noting that, when $V(x) = x^2$ (\ie~$\opH$ is the Davies operator), then $t_a = 2^{\frac13} a^{\frac16}$ and the above equation becomes
\begin{equation*}
	b(a) = \left(\frac32\right)^{\frac23} a^{\frac13} \left(\log(a^{\frac13} \eps^{-1})\right)^{\frac23} (1 + o(1)), \quad a \to +\infty,
\end{equation*}
(compare these curves with \cite[Eq.~7.5]{ArSi-resolvent-2022} for $n = 1$).

\subsubsection{The wave equation with unbounded damping}
\label{sssec:dwe.level.curves}
Let $\opT(\la)$ be the quadratic operator function in $\Lt(\R)$ described in \cite[Sec.~2.4]{arnal2022resolvent-dwe}, \ie
\begin{equation*}
	\opT(\la) = \Dt + q(x) + 2 \la a(x) + \la^2, \quad \Dom(\opT(\la)) = \Dom(\Dt + q) \cap \Dom(a),
\end{equation*}
with $\la \in \C \setminus (-\infty, 0]$. Assume that $a, q$ satisfy \cite[Asm.~3.1]{arnal2022resolvent-dwe} and let $\la := -c + i b$ with $c \in K \subset \overline \Rplus$ and $b \in \R \setminus \{0\}$ as in the statement of \cite[Thm.~3.5]{arnal2022resolvent-dwe}. Then we have (see \cite[Thm.~4.3]{arnal2022resolvent-dwe})
\begin{equation}
	\label{eq:Tla.res.norm.gral}
	\| \opT(\la)^{-1} \| = \| (\opA - c)^{-1} \| (2 |b|)^{-1} (1 + o(1)), \quad |b| \to +\infty,
\end{equation}
with
\begin{equation*}
	\opA = \Nt + a(x), \quad \Dom(\opA) = W^{1,2}(\R) \cap \Dom(a).
\end{equation*}
In \cite[Prop.~4.12]{arnal2022resolvent-dwe}, the above result was extended (with identical analytical expression for $\| \opT(\la)^{-1} \|$) to general curves adjacent to the imaginary axis
\begin{equation*}
	\la_b := - c(b) + i b,
\end{equation*}
with $c: \Rplus \to \Rplus$ satisfying suitable conditions (see \cite[Sub-sec.~4.2]{arnal2022resolvent-dwe}). As another application of \eqref{eq:A.resnorm}, we shall consider two examples of damping functions that satisfy Assumption~\ref{asm:W.airy}.

\begin{enumerate}[\upshape (i), wide]
	\item $a(x) = \log \langle x \rangle^p$, $p > 0$. From \eqref{eq:A.resnorm.ex.log}, we have
	\begin{equation*}
		\log\log(\| (\opA - c(b))^{-1} \|) = \frac{c(b)}{p} (1 + o(1)), \quad b \to +\infty.
	\end{equation*}
	Using \eqref{eq:Tla.res.norm.gral} with $\la = \la_b$ and substituting $\| (\opT(\la_b))^{-1} \| = \eps^{-1}$, with $\eps > 0$, we obtain the level curves
	\begin{equation}
		c(b) = p \log\log(2 b \eps^{-1}) (1 + o(1)), \quad b \to +\infty.
	\end{equation}
	It was shown in \cite[Sub-sec.~4.2]{arnal2022resolvent-dwe} that the admissible curves are those satisfying $0 < p < 1/2$.
	\item $a(x) = x^{2n}$, $n \in \N$. Applying \eqref{eq:A.resnorm.ex.poly} with $p = 2n$, we have
	\begin{equation}
		\label{eq:dwe.ga.res.norm}
		\| (\opA - c)^{-1} \| = \sqrt{\frac{\pi}{2n}} c^{\frac{1 - 2n}{4n}} \exp\left(\frac{4n}{2n + 1} c^{\frac{2n + 1}{2n}}\right) (1 + o(1)), \quad b \to +\infty.
	\end{equation}
	Combining \eqref{eq:Tla.res.norm.gral}, \eqref{eq:dwe.ga.res.norm} and substituting $\| (\opT(\la_b))^{-1} \| = \eps^{-1}$, with $\eps > 0$, it was \textit{conjectured} in \cite[Sub-sec.~4.2]{arnal2022resolvent-dwe} that the formula
	\begin{equation*}
		c(b) = \left(\frac{2n + 1}{4n}\right)^{\frac{2n}{2n + 1}} \left(\log(2 b \eps^{-1})\right)^{\frac{2n}{2n + 1}} (1 + o(1)), \quad b \to +\infty,
	\end{equation*}
	asymptotically describes the level curves in this case.
\end{enumerate}

\subsection{Optimality of resolvent bounds in \cite{Bandtlow-2004-267,SARIHAN2021QUANTITATIVE}}
\label{ssec:bandtlow.optimal}
Let $\opA$ satisfy the conditions in the statement of Theorem~\ref{thm:airy.norm}. Then $\sigma(\opA) = \emptyset$ and $\opA$ has compact resolvent (see Proposition~\ref{prop:airy.basic.properties}). It therefore follows that $\sigma(\opA^{-1}) = \{0\}$, \ie~$\opA^{-1}$ is quasi-nilpotent. For any $\la \in \C$, we can thus write
\begin{equation}
	\label{eq:bandtlow.res.decomp}
	\opA - \la = (\opI - \la \opA^{-1}) \opA \implies (\opA - \la)^{-1} = \opA^{-1} (\opI - \la \opA^{-1})^{-1}.
\end{equation}
Taking $W(x) = |x|^p$, $p > 0$, and $\opA_p= \Nt + |x|^p$ as in Example~\ref{ex:W.basic}~\ref{itm:W.poly.basic}, our next aim is to show that $\opA_p^{-1} \in \cS_{2r_p + \eps}$, with $r_p := (1 + p)/2p$ and $\eps > 0$ arbitrarily small.

The form
\begin{equation}
	h_p[f]:= \|f'\|^2 + \||x|^p f\|^2 + \|f\|^2 , \quad \Dom(h_p):=W^{1,2}(\R) \cap \Dom(|x|^p) 
\end{equation}	
defines the non-negative self-adjoint operator with compact resolvent
\begin{equation*}
	\opH_{2p} := \Dt + |x|^{2p} +1
\end{equation*}
and with $\sigma(\opH_{2p}) \subset [1, +\infty)$. By \eqref{eq:genAiry.graph.norm} and the second representation theorem \cite[Thm.~VI.2.23]{Kato-1966}, we have for any $u \in \Dom(\opA_p) = \Dom(\opH_{2p}^{\frac12}) = \Dom(h_p)$
\begin{equation*}
	\| \opA_p u \| + \| u \| \approx \| \Ntp u \| + \| |x|^p u \| + \| u \| \approx \| \opH_{2p}^{\frac12} u \|.
\end{equation*}
Hence $\opH_{2p}^{\frac12} \opA_p^{-1}$ is well-defined and moreover $\opH_{2p}^{\frac12} \opA_p^{-1} \in \sL(\Lt(\R))$. It is also known (see \eg~\cite{Titchmarsh-1954-5} or \cite[Prop.~6.1]{Mityagin-2019-139}) that the eigenvalues $\{\mu_k\}$ of the operator $\opH_{2p}$ satisfy
\begin{equation}
	\label{eq:H.ev}
	\mu_k = C_p k^{\frac1{r_p}} (1 + o(1)), \quad k \to +\infty,
\end{equation}
with $C_p > 0$ and $r_p$ as defined above. Observe that $\opH_{2p}^{-\frac12}$ is compact and furthermore, for any arbitrarily small $\eps > 0$ (with $s_k(\opH_{2p}^{-\frac12})$ denoting the singular values of $\opH_{2p}^{-\frac12}$ - refer to our notation in Section~\ref{sec:prelim}), we have
\begin{equation*}
	\sum_{k=1}^{\infty} (s_k(\opH_{2p}^{-\frac12}))^{2 r_p + \eps} \ls \sum_{k=1}^{\infty} k^{-\frac{2 r_p + \eps}{2 r_p}} < \infty.
\end{equation*}
This shows that $\opH_{2p}^{-\frac12} \in \cS_{2r_p + \eps}$ and, since $\opA_p^{-1} = \opH_{2p}^{-\frac12} \opH_{2p}^{\frac12} \opA_p^{-1}$ and $\| \opH_{2p}^{\frac12} \opA_p^{-1} \| < \infty$, we conclude that $\opA_p^{-1} \in \cS_{2r_p + \eps}$ (see \cite[Lem.~XI.9.9]{DS2}), as claimed. 

Recalling that $\opA_p^{-1}$ is also quasi-nilpotent, we can apply the Carleman-type estimate \cite[Thm.~3.4.6]{Ringrose-1971}, \cite[Cor.~XI.9.25]{DS2} or \cite[Thm.~2.1]{Bandtlow-2004-267} to establish (with $a_p > 0$)
\begin{equation}
	\label{eq:bandtlow.poly.res.est.1}
	\| (\opI - \la \opA_p^{-1})^{-1} \| \ls \exp(a_p \| \opA_p^{-1} \|_{2 r_p + \eps}^{2 r_p + \eps} |\la|^{2 r_p + \eps}).
\end{equation}
From \eqref{eq:bandtlow.res.decomp} and \eqref{eq:bandtlow.poly.res.est.1}, we deduce that
\begin{equation}
	\label{eq:bandtlow.poly.res.est.2}
	\log \| (\opA_p - \la)^{-1} \| \ls |\la|^{2 r_p + \eps}.
\end{equation}
Furthermore $2 r_p = (1 + p)/p$ and hence, comparing \eqref{eq:bandtlow.poly.res.est.2} with our estimate \eqref{eq:A.resnorm.ex.poly}, we conclude that this general upper bound yields an almost optimal rate for the growth of the resolvent of $\opA_p$.

\subsubsection{Improvement using the Schatten-Lorentz ideals in \cite{SARIHAN2021QUANTITATIVE}}
\label{sssec:S-L.ideals}
We note that, using the more general techniques presented in \cite{SARIHAN2021QUANTITATIVE}, it is possible to adapt the above reasoning to derive an upper bound for the resolvent of $\opA_p$ which is similar to \eqref{eq:bandtlow.poly.res.est.2} but with $\eps = 0$. 

The key is to observe that, due to the eigenvalue estimate \eqref{eq:H.ev}, the operator $\opH_{2p}^{-\frac12}$ belongs in fact to the Schatten-Lorentz ideal $\cS_{2r_p, \infty}$ and consequently so does $\opA_p^{-1}$ (see \cite[Prop.~3.9]{SARIHAN2021QUANTITATIVE}). One then applies \cite[Thm.~4.12]{SARIHAN2021QUANTITATIVE} to $\la \opA_p^{-1}$ noting firstly that, since $\opA_p^{-1}$ is quasi-nilpotent, the upper bound \cite[Eq.~10]{SARIHAN2021QUANTITATIVE} can be reduced to $\| (\opI - \la \opA_p^{-1})^{-1} \| \le F_{\overset{.}{\overline\omega}}(|\opA_p^{-1}|_{\overset{.}{\overline\omega}} |\la|)$, and secondly that, by \cite[Prop.~5.4]{SARIHAN2021QUANTITATIVE}, we also have $\log F_{\overset{.}{\overline\omega}}(|\la|) \sim |\la|^{2 r_p}$ as $|\la| \to +\infty$. We thus conclude
\begin{equation*}
	\log	\| (\opA_p - \la)^{-1} \| \ls |\la|^{2 r_p}, \quad |\la| \to +\infty.
\end{equation*}

\subsubsection{Schr\"odinger operators with imaginary monomial potentials}
Finally, although not using the results of this paper, we remark that the upper bounds in \cite[Thm.~4.1]{Bandtlow-2004-267} and \cite[Thm.~4.12]{SARIHAN2021QUANTITATIVE} for operators with non-empty spectrum are almost exhausted by the resolvents of Schr\"odinger operators
\begin{equation*}
	\opH_n=\Dt + \ii x^n, \qquad \Dom(\opH_n) = W^{2,2}(\R) \cap \Dom(x^n), \quad n \in \N \setminus \{1\}.
\end{equation*}
It is known that the spectrum of $\opH_n$ is discrete, $0 \notin \sigma(\opH_n)$ and
\begin{equation*}
	\sigma(\opH_{3}) \subset \Rplus, \quad \sigma(\opH_{2m}) \subset e^{i \frac\pi{2(m+1)}} \Rplus, \quad m \in \N;
\end{equation*}
see \cite{Dorey-2001-34,Shin-2002-229} for the case $n=3$ and employ complex scaling for $n = 2m$.

Using the graph norm separation of $\opH_n$ and \eqref{eq:H.ev}, one obtains that $\opH_n^{-1} \in \cS_{q_n,\infty}$ with $q_n = (n+2)/(2n)$, $n \in \N$. Thus, as in \eqref{eq:bandtlow.res.decomp}, writing
\begin{equation*}
	(\opH_n-\la)^{-1} = \opH_n^{-1} \left(I-\la \opH_n^{-1}\right)^{-1}, \quad \la \in \rho(\opH_n),
\end{equation*}
the application of \cite[Thm.~4.1]{Bandtlow-2004-267} yields
\begin{equation*}
	\begin{aligned}
		\|(\opH_n-\la)^{-1}\| &\ls
		\frac{\exp\left( \dfrac{c_n |\la|^{q_n+\eps}}{\dist(1,\sigma(\la \opH_n^{-1}))^{q_n+\eps}} \right)}{\dist(1,\sigma(\la \opH_n^{-1}))},  \quad \la \in \rho(\opH_n),
	\end{aligned}
\end{equation*}
where $c_n >0$ is independent of $\la$. Similarly as above in Sub-section~\ref{sssec:S-L.ideals}, the results \cite[Thm.~4.12, Prop.~5.4]{SARIHAN2021QUANTITATIVE} provide an improvement with $\eps =0$ which yields
\begin{equation}
\log \|(\opH_n-\la)^{-1}\| \ls |\la|^{q_n}, \quad |\la| \to \infty,
\end{equation}
when $\la$ is sufficiently far from the spectrum of $\opH_n$ so that $\dist(1,\sigma(\la \opH_n^{-1})) \gs 1$ (\eg~$\la$ diverges along a ray from where the eigenvalues lie).

On the other hand, pseudomode constructions, see \cite{Dencker-2004-57} (or \cite{Krejcirik-2015-56} with a detailed proof in one dimension), provide
\begin{equation}
\log \|(\opH_n - e^{\ii \theta} s)^{-1}\| \gs s^{q_n} , \quad s \to + \infty,
\end{equation}
where $\theta \in (0,\pi/2)$ if $n$ is even, $\theta \in (-\pi/2,\pi/2)$ if $n$ is odd.

\bibliography{references}
\bibliographystyle{acm}	

\end{document}